\documentclass[12pt]{amsart}
\usepackage{amscd,amsmath,amsthm,amssymb,graphics,centernot}

\usepackage{pgf,tikz}
\usetikzlibrary{arrows}

\def\frk{\frak}               

\def\Phi{{\frk n}}
\def\Phi{{\frk N}}
\def\MT{{\mathcal T}}

\def\MM{{\mathcal M}}

\def\xb{{\bold x}}

\def\opn#1#2{\def#1{\operatorname{#2}}} 
\opn\chara{char} \opn\length{\ell} \opn\pd{pd} \opn\rk{rk}
\opn\projdim{proj\,dim} \opn\injdim{inj\,dim} \opn\rank{rank}
\opn\depth{depth} \opn\grade{grade} \opn\height{height}
\opn\embdim{emb\,dim} \opn\codim{codim}

\opn\Tr{Tr} \opn\bigrank{big\,rank}
\opn\superheight{superheight}\opn\lcm{lcm}
\opn\trdeg{tr\,deg}
\opn\reg{reg} \opn\lreg{lreg} \opn\ini{in} \opn\lpd{lpd}
\opn\size{size}\opn\bigsize{bigsize}
\opn\cosize{cosize}\opn\bigcosize{bigcosize}
\opn\sdepth{sdepth}\opn\sreg{sreg}
\opn\link{link}\opn\fdepth{fdepth}
\opn\index{index}

\opn\div{div} \opn\Div{Div} \opn\cl{cl} \opn\Cl{Cl}
\opn\Spec{Spec} \opn\Supp{Supp} \opn\supp{supp} \opn\Sing{Sing}
\opn\Ass{Ass} \opn\Min{Min}\opn\Mon{Mon} \opn\dstab{dstab} \opn\astab{astab}
\opn\Syz{Syz}
\opn\Ann{Ann} \opn\Rad{Rad} \opn\Soc{Soc}
\opn\Im{Im} \opn\Ker{Ker} \opn\Coker{Coker} \opn\Am{Am}
\opn\Hom{Hom} \opn\Tor{Tor} \opn\Ext{Ext} \opn\End{End}
\opn\Aut{Aut} \opn\id{id}

\opn\nat{nat}
\opn\pff{pf}
\opn\Pf{Pf} \opn\GL{GL} \opn\SL{SL} \opn\mod{mod} \opn\ord{ord}
\opn\Gin{Gin} \opn\Hilb{Hilb}\opn\sort{sort}
\opn\initial{init}
\opn\ende{end}
\opn\height{height}
\opn\type{type}
\opn\aff{aff} \opn\con{conv} \opn\relint{relint} \opn\st{st}
\opn\lk{lk} \opn\cn{cn} \opn\core{core} \opn\vol{vol}
\opn\link{link} \opn\star{star}\opn\lex{lex}
\opn\gr{gr}

\def\pot#1#2{#1[\kern-0.28ex[#2]\kern-0.28ex]}

\opn\dirlim{\underrightarrow{\lim}}
\opn\inivlim{\underleftarrow{\lim}}
\let\union=\cup
\let\sect=\cap

\def\Implies{\ifmmode\Longrightarrow \else
        \unskip${}\Longrightarrow{}$\ignorespaces\fi}
\def\implies{\ifmmode\Rightarrow \else
        \unskip${}\Rightarrow{}$\ignorespaces\fi}
\def\iff{\ifmmode\Longleftrightarrow \else
        \unskip${}\Longleftrightarrow{}$\ignorespaces\fi}

\let\:=\colon
\newtheorem{Theorem}{Theorem}[section]
 \newtheorem{Lemma}[Theorem]{Lemma}
 \newtheorem{Corollary}[Theorem]{Corollary}
 \newtheorem{Proposition}[Theorem]{Proposition}

 \newtheorem{Examples}[Theorem]{Examples}

 \newtheorem{Conjecture}[Theorem]{Conjecture}

\let\epsilon\varepsilon
\let\kappa=\varkappa
\def\qed{\ifhmode\textqed\fi
      \ifmmode\ifinner\quad\qedsymbol\else\dispqed\fi\fi}
\def\textqed{\unskip\nobreak\penalty50
       \hskip2em\hbox{}\nobreak\hfil\qedsymbol
       \parfillskip=0pt \finalhyphendemerits=0}
\def\dispqed{\rlap{\qquad\qedsymbol}}

\opn\dis{dis}
\def\pnt{{\raise0.5mm\hbox{\large\bf.}}}

\opn\Lex{Lex}

\begin{document}

 \title{On the index of powers of edge ideals}

 \author {Mina Bigdeli, J\"urgen Herzog and Rashid Zaare-Nahandi}

\address{Mina Bigdeli, Department  of Mathematics,  Institute for Advanced Studies in Basic Sciences (IASBS),
45195-1159 Zanjan, Iran} \email{mina.bigdeli@yahoo.com}

\address{J\"urgen Herzog, Fachbereich Mathematik, Universit\"at Duisburg-Essen, Campus Essen, 45117
Essen, Germany} \email{juergen.herzog@uni-essen.de}

\address{Rashid Zaare-Nahandi, Department of Mathematics,  Institute for Advanced Studies in Basic Sciences (IASBS),
45195-1159 Zanjan, Iran} \email{rashidzn@iasbs.ac.ir}

\thanks{The paper was written while the first author was visiting the Department of Mathematics of University Duisburg-Essen. She wants to express her thanks for its hospitality.}

 \begin{abstract}
The index of a graded ideal measures the number of linear steps in the graded minimal free resolution of the ideal. In this paper we study the index of powers and squarefree powers of edge ideals. Our results indicate that the index as a function of the power of an edge ideal $I$ is strictly increasing if $I$ is linearly presented. Examples show that this needs not to be the case for monomial ideals generated in degree greater than two.
 \end{abstract}

\subjclass[2010]{Primary 13D02, 13C13; Secondary 05E40.}
\keywords{edge ideals, index,  powers of ideals, resolutions,}

 \maketitle

\section*{Introduction}
In recent years the study of algebraic and homological properties of powers of ideals has been one of the main subjects of research in  Commutative Algebra. Generally speaking many of those properties, like for example depth, projective dimension or regularity  stabilize for large powers  (see \cite{B}, \cite{Ca}, \cite{Ch}, \cite{CHT}, \cite{Co}, \cite{HH}, \cite{HW}, \cite{HHZ1}, \cite{HHZ2}), while their  initial behavior is  often quite mysterious, even for monomial ideals. However with many respects monomial ideals generated in degree 2 behave more controllable from the very beginning. So now let $I$ be a monomial ideal generated in degree  $2$. The second author together with Hibi and Zheng showed in \cite{HHZ2} that if  $I$ has a linear resolution, then all of its powers have a linear resolution as well. More recently there have been several interesting  generalizations  of  this result. In case that $I$ is  squarefree, $I$ may be viewed as the edge ideal  of a finite simple graph $G$, and  in this case  Francisco, H$\grave{a}$ and Van Tuyl raised the question whether $I^k$ has a linear resolution for $k\geq 2$, assuming  the complementary graph contains no induced $4$-cycle, equivalently, $G$ is gap free. However, Nevo and Peeva showed by an example  \cite[Counterexample 1.10]{NP} that this is not always the case. On the other hand, Nevo \cite{N} showed that $I^2$ has a linear resolution if $G$ is gap and claw  free, and Banerjee \cite{B} gives a positive answer to the above question  under the  additional assumption that $G$ is gap and cricket free. Here we should note that claw free implies cricket free.

In this paper we attempt to generalize the result of Hibi, Zheng and the second author of this  paper in a different direction. An  ideal $I$ is called $r$ steps linear, if $I$ has a linear resolution up to homological degree $r$. In other words, if $I$ is generated in a single degree, say $d$, and
 $\beta_{i,i+j}(I)=0$ for all pairs $(i,j)$ with  $0\leq i\leq r$ and $j>d$. The number
\[
\index(I)=\sup\{r\: \text{$I$ is $r$ steps linear}\}+1
\]
is called the  index of $I$.  A related invariant, called the $N_{d,r}$--property, was  first considered by Green and Lazarsfeld in \cite{GL1}, \cite{GL2}.  In the paper  \cite{BCR} by Bruns et~al.~the Green-Lazarsfeld index was introduced for quadratically generated ideals as the largest integer $r$ such that the $N_{2,r}$--property holds. We use the same terminology applied to any graded ideal in the polynomial ring and  call it simply the index of the ideal.

The main result of Section 2  (Theorem~\ref{main}) is the following: Let $I$ be a monomial ideal generated in degree $2$. We interpret  $I$ as the edge ideal  of a graph $G$  which may also have loops (corresponding to squares among the monomial generators of $I$).  Then the following conditions are equivalent: (a) $G$ is gap free, i.e. no induced subgraph of $G$ consists of two disjoint edges; (b) $\index(I^k)>1$ for all $k$; (c) $\index(I^k)>1$ for some $k$.

Theorem~\ref{main} is not valid  for monomial ideals generated in  degree $>2$. There is an example by Conca \cite{Co} of a monomial  ideal $I$ generated in degree $3$  with linear resolution, that is, $\index(I)=\infty$,  and  with the property that  $\index(I^2)=1$.

Theorem~\ref{main} implies in particular that for a monomial ideal generated in degree $2$ we have     $\index(I)=1$ if and only if $\index(I^k)=1$ for all $k$ . Again this fails if $I$ is not generated in degree $2$. Indeed, for $n\geq 4$ consider the ideal $I=(x^n,x^{n-1}y,y^{n-1}x,y^n)$. Then $\index(I^k)=1$ for $k=1,\ldots,n-3$ and $\index(I ^k)=\infty$  for $k>n-3$. There are also many such counterexamples of monomial ideals generated in degree $3$.

The ideal $I$ in the example of Nevo and Peeva has index 2, its square has index 7,  while $I^3$ and $I^4$ have a linear resolution. This example and other experimental evidence lead us to make the following
\begin{Conjecture}
Let $I$ be a monomial ideal  generated in degree $2$  with linear presentation.  Then $\index(I^{k+1})>\index(I^k)$ for all
$k$. Here we use the convention that $\infty>\infty$. 
\end{Conjecture}
This conjecture implies that $\index(I^k)>k$ if  $\index(I)>1$. In particular, for a gap free graph  $G$,  this would imply that $I(G)^k$ has a linear resolution for $k>n-2$.

For the proof of our Theorem~\ref{main} we use the  theory of $\lcm$-lattices introduced by Gasharov, Peeva and  Welker \cite{GPW}. As an easy application of their theory the monomial ideals of index $>1$ can be characterized by the fact that certain graphs associated with such ideals  are connected. This criterion is used in the proof of Theorem~\ref{main}.

If the index of a graded ideal is finite, then it is at most its projective dimension. In the case that $\index(I)=\projdim(I)$ we say that $I$ has maximal finite index. In Section
~3 edge ideals of maximal finite index are classified. They turn out to be the edge ideals of the complement of a cycle, see Theorem~\ref{main2}. The essential
tools to prove this result are Hochster's formula to compute the graded Betti numbers of a squarefree monomial ideal as well as the result of  \cite[Theorem 2.1]{EGHP} in which the index of an edge ideal is characterized in terms  of  the underlying graph.  As a consequence of Theorem~\ref{main2} it is shown in Corollary~\ref{maxlinear} that all powers $I^k$ for $k\geq 2$ have a linear resolution for an ideal of maximal finite index $>1$. This supports our conjecture that the index of the powers $I^k$ of an edge ideal $I$ is a strictly increasing function on $k$.

Our final Section~4 is devoted to the study of the index of the squarefree powers of edge ideals. The index of squarefree powers shows a quite different behavior than that of ordinary powers.  Let $I$ be the edge ideal of a finite graph $G$. We denote  the $k$-th squarefree power of $I$ by $I^{[k]}$. It is clear that the unique minimal monomial set of generators of $I^{[k]}$ corresponds to the matchings of $G$ of size $k$. In particular, if $\nu(G)$ denotes the matching number of $G$, that is maximal size of a matching of $G$, then  $\nu(G)$  coincides with the maximal number $k$ such that $I^{[k]}\neq 0$. In Theorem~\ref{lin.quo} we show  that $I^{[\nu(G)]}$ always has linear quotients. In particular $\index(I^{[\nu(G)]})=\infty$ no matter whether or not $\index(I)=1$. A matching with the property that one edge of the matching forms a gap with any other edge of the matching will be called a  restricted matching. We denote by $\nu_0(G)$ the maximal size of a restricted matching of $G$. If there is no restricted matching  we set $\nu_0(G)=1$. There are examples which show that $\nu(G)-\nu_0(G)$ may be arbitrary large. However for trees one can see that $\nu_0(G)\geq \nu(G)-1$. It is shown in Lemma~\ref{levico} that $\index(I^{[k]})=1$ for $k<\nu_0(G)$, and we conjecture that $\index(I^{[k]})>1$ for all $k\geq \nu_0(G)$ and prove this conjecture in Theorem~\ref{cycle} for any cycle.

\label{1}\section{Monomial ideals with index $>1$.}

Let $K$ be a field, $S=K[x_1,\ldots,x_n]$ the polynomial ring over $K$ in $n$ indeterminates, and let $I\subset S$ be a monomial ideal generated in degree $d$.

The ideal is called {\em $r$ steps linear}, if $I$ has a linear resolution up to homological degree $r$, in other words, if
 $\beta_{i,i+j}(I)=0$ for all pairs $(i,j)$ with  $0\leq i\leq r$ and $j>d$. Then the number
\[
\index(I)=\sup\{r\: \text{$I$ is $r$ steps linear}\}+1
\]
is called the {\em  index}  of $I$. In particular, $I$ has a linear resolution if and only if $\index(I)=\infty$.
A monomial ideal $I$ of finite index has $\index(I)\leq \projdim(I)$. We say that  $I$ has {\em maximal finite index} if equality holds.

In this section we derive a criterion for a monomial ideal $I$ to be linearly presented, i.e. $\index(I)>1$. This criterion is actually an immediate consequence of the lcm-lattice formula \cite{GPW} by Gasharov, Peeva and Welker for the multi-graded Betti numbers of a monomial ideal $I$, not necessarily  generated in a single degree.

Using the results of this section, we give in Section~\ref{112} a characterization of the  finite graphs whose all powers of  edge ideals are linearly presented. These graphs turn out to be gap free. We say a graph $G$ is {\em gap free} if for any two disjoint  edges $e,e'\in E(G)$ there exists an edge $f\in E(G)$ such that $e\cap f\neq\emptyset\neq e'\cap f$. In the case that $G$ is simple,  $G$ is gap free if and only if its complement $\bar{G}$ has no 4-cycle.

Let $G(I) = \{u_1,\ldots,u_m\}$ be the unique minimal set of  monomial generators of $I$. We denote by  $L(I)$ the lcm-lattice of $I$, i.e.\ the poset whose elements  are labeled by the least
common multiples of subsets of monomials in $G(I)$ ordered by divisibility. The unique minimal element in $L(I)$ is $1$. For any $u\in L(I)$ we denote by $(1,u)$ the open interval of $L(I)$ which by definition is  the induced subposet of $L(I)$ with elements $v\in L(I)$ with $1<v<u$. Furthermore,  we denote by $\Delta((1,u))$ the order complex of the poset $(1,u)$.

The minimal graded free resolution of $I$ is multi-graded. Identifying a monomial with its multi-degree we denote the multi-graded Betti numbers of $I$ by $\beta_{i,u}(I)$, where $i$ is the homological degree and $u$ is a monomial. By Gasharov, Peeva and Welker one has
\begin{eqnarray}
\label{lcm}
\beta_{i,u}(I)=\dim_K \widetilde{H}_{i-1}(\Delta((1,u)); K) \quad \text{for all $i\geq 0$ and all $u\in L(I)$.}
\end{eqnarray}
Moreover, $\beta_{i,u}(I)=0$ if $u\not \in L(I)$.

\medskip
Now suppose that all generators of $I$ are of degree $d$. We define the graph $G_I$ whose vertex set is $G(I)$ and for which  $\{u,v\}$ is an edge of $G_I$ if and only if $\deg (\lcm(u,v))=d+1$.

For all $u, v\in G(I)$ let  $G^{(u,v)}_I$ be the induced subgraph of $G_I$ with vertex set
\[
V(G^{(u,v)}_I)=\{w\: \text{$w$ divides $\lcm(u, v)$}\}.
\]
\begin{Proposition}
\label{graph}
\label{criterion}
Let $I$ be a monomial ideal generated in degree $d$.  Then $I$ is linearly presented if and only if $G^{(u,v)}_I$ is connected for all $u, v\in G(I)$.
\end{Proposition}

\begin{proof}
By (\ref{lcm}) the ideal $I$ is linearly presented if and only if all the open intervals $(1,w)$ with $w\in L(I)$ and $\deg w>d+1$ are connected. Considering the  Taylor complex of $I$ we see that $\beta_{1,w}(I)=0$ if there exists no $u ,v \in G(I)$ such that $w=\lcm(u,v)$. Thus we need only to consider intervals  $(1,w)$ with $w=\lcm(u,v)$ for some $u ,v \in G(I)$, and hence   $I$ is linearly presented if and only if  $\Delta((1,\lcm(u, v)))$ is connected  for all $u,v\in G(I)$ with $\deg(\lcm(u,v))>d+1$.

We first assume that $G_I^{(u,v)}$ is connected for all  $u,v \in G(I)$. Now let  $u,v \in G(I)$ with $\deg(\lcm(u,v))>d+1$. Let $C$ and $C'$ be  maximal chains of the interval $(1,\lcm(u,v))$ (i.e.\  facets of $\Delta((1,\lcm(u,v)))$). For a chain $D$ in $(1,\lcm(u,v))$  we denote by $\min(D)$ the minimal element in $D$. Obviously, $\min(D)\in V(G_I^{(u,v)})$.  Let $w=\min(C)$ and $w'=\min(C')$. Then $w,w'\in V(G^{(u,v)}_I)$. Hence there exists a sequence $w_{1},\ldots,w_{r}\in V(G^{(u,v)}_I)$ with $w=w_{1}$ and $w'=w_{r}$ and such that the degree of $v_j:=\lcm(w_{j},w_{j+1})$ is $d+1$ for $j=1,\ldots,r-1$. Since $v_j$ divides $\lcm(u,v)$ and since $\deg v_j<\deg (\lcm(u,v))$ it follows that $v_j\in (1,\lcm(u,v))$. Thus there  exist maximal chains $C_j$ and $D_j$ with $w_{j},v_j\in C_j$ and $v_j,w_{j+1}\in D_j$. Consider the sequence of maximal chains
\[
C,C_1,D_1,C_2,D_2,\ldots,C_{r-1},D_{r-1}, C'.
\]
By construction any two successive chains in this sequence have a non-trivial intersection.   This shows that $\Delta((1,\lcm(u,v)))$ is connected.

Conversely, assume that $\Delta((1,\lcm(u,v)))$ is connected for all $u,v \in G(I)$ with $\deg(\lcm(u,v))>d+1$. By induction on $\deg (\lcm(u,v))$ we prove that $G_I^{(u,v)}$ is connected for all $u,v \in G(I)$ with $\deg(\lcm(u,v))>d+1$.

In order to prove this, let  $u,v \in G(I)$ with $\deg(\lcm(u,v))>d+1$, and let $w,w'\in G^{(u,v)}_I$ with $w\neq w'$. There exist maximal chains $C$ and $D$  in $(1,\lcm(u,v))$ with $\min(C)=w$ and $\min(D)=w'$. Since $\Delta((1,\lcm(u,v)))$ is connected, there exist maximal chains $C_1,\ldots,C_r$ in  $\Delta((1,\lcm(u,v)))$ with $C=C_1$ and $C_r=D$ and such that $C_j\sect C_{j+1}\neq \emptyset$  for $j=1,\ldots,r-1$. Let $w_{j}=\min(C_j)$ for $j=1,\ldots,r$. Let  $j$ be such that $w_{j}\neq w_{j+1}$. Then $d+1 \leq \deg (\lcm(w_{j}, w_{j+1}))<\deg (\lcm(u,v))$ because  $\lcm(w_{j}, w_{j+1})$ divides $\lcm(u,v)$, and $\lcm(w_{j}, w_{j+1})\in (1,\lcm(u,v))$ because $C_j\sect C_{j+1}\neq \emptyset$. If $\deg(\lcm(u,v))=d+2$, it follows that  $\deg (\lcm(w_{j}, w_{j+1}))=d+1$ for all $j$  with $w_{j}\neq w_{j+1}$. This shows that  $G_I^{(u,v)}$ is connected whenever $\deg(\lcm(u,v))=d+2$ and establishes the proof of the induction begin.

Suppose now that  $\deg (\lcm(u,v))>d+2$.
Since  $\Delta((1,\lcm(w_{j},w_{j+1})))$ is connected and $\deg (\lcm(w_{j}, w_{{j+1}}))<\deg (\lcm(u,v))$ we may apply our induction hypothesis and deduce that $w_{j}$ and $w_{j+1}$ are connected in $G_I^{(w_{j},w_{j+1})}$. Since $G_I^{(w_{j},w_{j+1})}$ is an induced  subgraph of $G_I^{(u,v)}$  it follows that $w_{j}$ and $w_{{j+1}}$ are also connected in $G_I^{(u,v)}$.  Finally, since $w=w_{1}$ and $w'=w_{r}$. It follows that $G_I^{(u,v)}$ is connected for all $u,v\in G(I)$ with $\deg(\lcm(u,v))>d+1$. If $\deg(\lcm(u,v))\leq d+1$, then $G_I^{(u,v)}$ is obviously connected.
\end{proof}

\begin{Corollary}
\label{connected}
Let $I$ be a monomial ideal generated in degree $d$. Then $I$  is linearly presented if and only if for  all $u, v\in G(I)$  there is a path in $G^{(u,v)}_I$ connecting $u$ and $v$.
\end{Corollary}

\begin{proof}
Because of Proposition~\ref{graph} it suffices to show  that the following statements are equivalent:
\begin{enumerate}
\item[(i)]  $G^{(u,v)}_I$ is connected for all $u, v\in G(I)$;
\item[(ii)] for all  $u, v\in G(I)$,  there is a path in $G^{(u,v)}_I$  connecting $u$ and $v$.
\end{enumerate}

(i)\implies (ii) is obvious.

(ii)\implies (i):  Let $w\in G^{(u,v)}_I$ with $w\neq u$.  It is enough to show that $w$ is connected to $u$ by a path in $G^{(u,v)}_I$. By assumption $w$ is connected to $u$ by a path in $G^{(u,w)}_I$. Since $w\in G^{(u,v)}_I$ it follows that $\lcm(u,w)$ divides $\lcm(u,v)$. This implies that $G^{(u,w)}_I$ is an  induced subgraph of $G^{(u,v)}_I$. Thus the path connecting $w$ with $u$ in $G^{(u,w)}_I$ also connects $w$ and $u$  in $G^{(u,v)}_I$.
\end{proof}

Note that, in general, connectedness condition of each subgraph $G_I^{(u,v)}$ given in Corollary \ref{connected} can not be replaced with the connectedness of the graph $G_I$. Let $I=(x^4, x^3y, x^3z, x^2y^2,  x^2z^2, xy^3, xz^3, y^4, y^3z, y^2z^2, yz^3, z^4)\subset K[x,y,z]$. Then $G_I$ is connected, while $I$ is not linearly presented. Indeed, there is no path between $x^2y^2$ and $x^2z^2$ in $G_I^{(x^2y^2,x^2z^2)}$.

\section{Powers of edge ideals of index $>1$}\label{112}
Let $\MM$ be the set of all monomial ideals of $S$ generated in degree two and $\MT$ be the set of all graphs on the vertex set $[n]$  which do not  have double edges but may have loops. There is an obvious  bijection between $\MM$ and $\MT$. Indeed, if $I\in \MM$ then  the graph $G\in \MT$ corresponding to $I$ has the edge set $E(G)=\{\{i,j\}\: x_ix_j\in G(I)\}$. In case $i=j$,  the corresponding edge is a loop.


\begin{Theorem}
\label{main}
Let $G$ be a finite graph (possibly with loops) and  let $I$ be its edge ideal. The following conditions are equivalent:
\begin{itemize}
\item[(a)] $G$ is gap free;
\item[(b)] $\index (I^{k})>1$ for all $k\geq 1$, i.e. $I^k$ is linearly presented for all $k\geq 1$;
\item[(c)] $\index(I^{k})>1$ for some $k\geq 1$, i.e.  $I^k$ is linearly presented for some $k\geq 1$.
\end{itemize}
\end{Theorem}

\begin{proof}
(a)\implies (b):  Note that in case $k=1$ the equivalence of (a) and (b) has been proved in \cite[Corollary~2.9]{DHS}. Here we give a direct proof for the more general case: first we show that if $G$ is gap free, then $I$ is linearly presented. Then we prove that if  an edge ideal $I$ is linearly presented, then all its powers are linearly presented as well.

Using Corollary~\ref{connected}, to show that $I$ is linearly presented,  it is enough to prove that  for all $u, v\in G(I)$  there is a path in  $G^{(u,v)}_{I}$ connecting $u$ and $v$. 
Hence   $u=x_ix_j$ and $v=x_{i'}x_{j'}$ with $\{i, j\}, \{i', j'\}$ edges in $G$.  If $\deg(\lcm(u,v))=3$, then, by definition, $\{u,v\}\in E(G^{(u,v)}_{I})$ and so we are done. Suppose $\deg (\lcm(u,v))=4$. Our assumption implies that at least one of the edges  $\{i, i'\}, \{i, j'\}, \{i', j\}, \{i', j'\}$ is in $E(G)$. Without loss of generality we may assume that $\{i, i'\}\in E(G)$. Set $w=x_ix_{i'}$. So $w\in V(G^{(u,v)}_{I})$ and $w$ connects $u$ to $v$. Therefore $I$ is linearly presented.

\medskip
Now we prove that if $I$ is linearly presented, then $I^k$ is linearly presented for all $k\geq 1$. Using Corollary~\ref{connected}, it is enough to show that  for all $u, v\in G(I^k)$  there is a path in  $G^{(u,v)}_{I^k}$ connecting $u$ and $v$. We prove this by induction on $k$. Since $I$ is linearly presented   there is a path in  $G^{(u,v)}_{I}$ connecting $u$ and $v$ for all $u, v\in G(I)$. Let $k>1$, and suppose that $G^{(u,v)}_{I^{k-1}}$ is connected for all $u, v\in G(I^{k-1})$ with  $\deg(\lcm(u,v))>2(k-1)+1$.

 Assume that    $u/\tilde{w}, v/\tilde{w}\in G(I^{k-1})$ for some $\tilde{w}\in G(I)$. Since $\deg (\lcm(u/\tilde{w}, v/\tilde{w}))> 2(k-1)+1$, our  induction hypothesis  implies that there is a path  $w_0, w_1, \ldots, w_r$ in $G^{({u}/{\tilde{w}},{v}/{\tilde{w}})}_{I^{k-1}}$ with   $w_0={u}/{\tilde{w}}$ and $w_r={v}/{\tilde{w}}$. Since  $\deg (\lcm(w_i,w_{i+1})) =2(k-1)+1$ it follows that $\deg (\lcm(\tilde{w}w_i, \tilde{w}w_{i+1}))=2k+1$ for all $0\leq i\leq r-1$, and since  $\tilde{w}w_j\in V(G^{(u,v)}_{I^{k}})$ for all $j$,   the sequence $u=\tilde{w}w_0, \tilde{w}w_1, \ldots, \tilde{w}w_r=v $ is a path in $G^{(u,v)}_{I^{k}}$ connecting   $u$ and $v$.

We may now  suppose that  $u/\tilde{w}, v/\tilde{w}\notin G(I^{k-1})$ for all $\tilde{w}\in G(I)$.
Since $u\neq v$ and $\deg u=\deg v$, there is an index $ i$ with  $\deg_{x_i}v>\deg_{x_i}u$. In particular, there exists $\tilde{v}\in G(I)$ such  that $v/\tilde{v}\in G(I^{k-1})$ and $\tilde{v}=x_ix_j$ for some  $j$. In the further discussions we will distinguish four cases.
\begin{itemize}
\item[(i)] $\deg_{x_i}u\neq 0$ and $\deg_{x_j}u\neq 0$,
\item[(ii)] $\deg_{x_i}u\neq 0$ and $\deg_{x_j}u= 0$,
\item[(iii)] $\deg_{x_i}u= 0$ and $\deg_{x_j}u\neq 0$,
\item[(iv)] $\deg_{x_i}u= 0$ and $\deg_{x_j}u= 0$.
\end{itemize}
 We now first consider the cases (i), (ii) and (iii) and construct in these cases $w\in V(G^{(u,v)}_{I^{k}})$ such that  following conditions hold:
\begin{itemize}
\item[($\alpha$)] $\deg (\lcm(u,w)) =2k+1$;
\item[($\beta$)] $w/\tilde{v}\in G({I^{k-1}})$.
\end{itemize}
Condition ($\alpha$) implies that $\{u,w\}\in E(G^{(u,v)}_{I^{k}})$.  In the case that $\deg(\lcm(w,v))\leq 2k+1$,  $w$ is connected to $v$ in $G^{(u,v)}_{I^{k}}$, and so $u$ and $v$ are connected in $G^{(u,v)}_{I^{k}}$. If   $\deg(\lcm(w,v))>2k+1$, then condition ($\beta$)   allows us  to use  induction on $k$ as before,  and to conclude that $w$ is connected to $v$ by a path in  $G^{(u,v)}_{I^{k}}$, and hence $u$ and $v$ are connected in $G^{(u,v)}_{I^{k}}$.  Thus ($\alpha$) together with  ($\beta$) implies that there is a path in $G^{(u,v)}_{I^{k}}$ connecting $u$ and $v$.

\medskip
There exists a factor $\tilde{u}\in G(I)$  of $u$ such that $u/\tilde{u}\in G(I^{k-1})$ which in the cases (i) and (iii) is of the form $x_jx_{i_1}$ for some $i_1$  and in case (ii) is of the form $x_ix_{i_2}$ for some $i_2$.  It is seen that $\tilde{v}\neq \tilde{u}$, since otherwise $u/\tilde{v}, v/\tilde{v}\in G(I^{k-1})$, a contradiction. It follows that $i_1\neq i$ and $i_2\neq j$.

Let $w=(u/\tilde{u})\tilde{v}$.
Then $w\in G(I^k)$.

In case (i), $\deg_{x_i}w= \deg_{x_i}u+1$. Since  $\deg_{x_i}v>\deg_{x_i}u$, it follows  that $\deg_{x_i}w\leq \deg_{x_i}v$. We also note that $\deg_{x_j}w=\deg_{x_j}u$ and $\deg_{x_{i_1}}w=\deg_{x_{i_1}}u-1\leq \deg_{x_{i_1}}u$.

 In case (ii), $\deg_{x_i}w= \deg_{x_i}u$. Moreover, $\deg_{x_j}w=\deg_{x_j}u+1=1$  because $x_j$ does not divide $u$. However, since $x_j$ divides $v$, it follows that $\deg_{x_j}w\leq \deg _{x_j}v$. Finally   $\deg_{x_{i_2}}w=\deg_{x_{i_2}}u-1\leq \deg_{x_{i_2}}u$.

In case (iii), $\deg_{x_i}w= \deg_{x_i}u+1=1$ because $x_i$ does not divide $u$. However, since $x_i$ divides $v$, it follows that $\deg_{x_i}w\leq \deg _{x_i}v$.    Moreover, $\deg_{x_j}w=\deg_{x_j}u$ and $\deg_{x_{i_1}}w=\deg_{x_{i_1}}u-1\leq \deg_{x_{i_1}}u$.

Thus in all the three  cases  $\deg_{x_t}w \leq\deg_{x_t}(\lcm (u, v))$ for all  variables $x_t$. Therefore $w$ divides $\lcm(u,v)$, and so by definition $w\in V(G^{(u,v)}_{I^{k}})$.

Note that $\deg (\lcm(u, w))=2k+1$, and $w/\tilde{v}=u/\tilde{u}$ which implies that $w/\tilde{v}\in G({I^{k-1}})$. Therefore the assertion follows in these three cases.

\medskip
 Now we consider  case (iv).  Let $\tilde{u}\in G(I)$ with  $u/\tilde{u}\in G(I^{k-1})$, and let $w=(u/\tilde{u})\tilde{v}$ with $\tilde{v}$ as above. Then $w\in G(I^k)$. Since neither $x_i$ nor $x_j$ divides $u$, we have $\deg_{x_i}w=1=\deg_{x_j}w$. Thus $\deg_{x_t}w \leq\deg_{x_t}(\lcm (u, v))$ for all  variables $x_t$. It follows that $w$ divides $\lcm(u,v)$ and so $w\in V(G^{(u,v)}_{I^{k}})$.

 Moreover, $w\neq v$. Indeed,  suppose that  $w= v$. Then $v/\tilde{v}=u/\tilde{u}$.  Let $\tilde{w}\in G(I)$ with $(v/\tilde{v})/\tilde{w}\in G(I^{k-2})$. Then  $v/\tilde{w}, u/\tilde{w}\in G({I^{k-1}})$, a contradiction.

 Furthermore, $w/\tilde{v}, v/\tilde{v}\in G({I^{k-1}})$, and so if $\deg(\lcm(w,v))>2k+1$, by using the induction hypothesis there exists a path between $w/\tilde{v}$ and $v/\tilde{v}$ in $G^{(w/\tilde{v},v/\tilde{v})}_{I^{k-1}}$. As above this implies that $v$ and $w$ are connected in $G^{(w,v)}_{I^{k}}$ and hence in $G^{(u,v)}_{I^{k}}$, since $G^{(w,v)}_{I^{k}}$ is a subgraph of $G^{(u,v)}_{I^{k}}$. In the case that $\deg(\lcm(w,v))\leq 2k+1$, it is obvious that $v$ and $w$ are connected in $G^{(u,v)}_{I^{k}}$. Also by construction of $w$, we have $\deg(\lcm(u,w))>2k+1$, and the monomials $u$ and $w$ have a common factor, say $\tilde{w}\in G(I)$ such that  $w/\tilde{w}, u/\tilde{w}\in G({I^{k-1}})$. Again by using our induction hypothesis, we conclude that there exists a path between $w$ and $u$ in $G^{(u,v)}_{I^{k}}$.
 Therefore $u$ and $v$ are connected in $G^{(u,v)}_{I^{k}}$ also in this case.

\medskip
(b)\implies (c) is  obvious.

\medskip
(c)\implies (a):  Assume that $G$  is not gap free. Thus there exist four different vertices $i, i', j, j'$ such that   $\{i, i'\}, \{j, j'\}\in E(G)$ while the edges $\{i, j \}, \{i, j'\}, \{i', j\}$ and $\{i', j'\}$ are not in $E(G)$. Let $k>0$ be an arbitrary integer. We will show that $u=(x_{i}x_{i'})^k$ and $v=(x_{j}x_{j'})^k$ are not connected in $G_{I^k}^{(u, v)}$. Suppose that  they are connected. Then there  exists a monomial $w\in V(G_{I^k}^{(u, v)})$ such that $\deg (\lcm(u,w))=2k+1$.  Clearly, $w\in V(G_{I^k}^{(u, v)})$  yields that $w\in G(I^k)$  with the property that $w$ divides $\lcm(u,v)$. Moreover, $\deg (\lcm(u,w))=2k+1$ implies that either $x_i^kx_{i'}^{k-1}$ or $x_i^{k-1}x_{i'}^{k}$ divides $w$.

Note that $\lcm(u, v)=x_i^kx_{i'}^kx_{j}^kx_{j'}^k$ and since  $w$ divides $\lcm(u,v)$, either $x_j$ or $x_{j'}$ divides $w$.  This means that one of $\{i, j \}$, $\{i, j'\}$, $\{i', j\}$ or $\{i', j'\}$ must be an edge of $G$ which is a contradiction. Therefore, $u$ and $v$ are not connected in $G_{I^k}^{(u, v)}$. Corollary~\ref{connected} implies that $I^k$ is not linearly presented, a contradiction.
\end{proof}

\begin{Examples}\rm
(a) Let $G$ be a tree  on the vertex set $[n]$ and let $I$ be its edge ideal. Then either $\index(I^k)=1$ or $\index(I^k)=\infty$ for any $k>0$. Indeed, suppose that   $\index(I)=t<\infty$.  If $n\leq 4$, 
then either $\height(I)=1$ or $I=(x_1x_2,x_2x_3,x_3x_4)$. In both cases it is clear that $I$ has a linear resolution.
 Now let $n>4$. By \cite[Theorem 2.1]{EGHP} there exists a minimal cycle $C$ of length $t+3$ in $\bar{G}$. Suppose that $V(C)=\{1, 2, \ldots, t+3\}$ and $E(C)=\{\{i, i+1\}\:\ 1\leq i\leq t+2\}\cup\{\{1, t+3\}\}$.   If $t>1$, then $|V(C)|\geq 5$ and since $C$ is minimal we have $\{1,3\}, \{1,4\}, \{2,4\}, \{2,5\}, \{3,5\}\notin E(\bar{G})$. Therefore there exists a cycle  in $G$, a contradiction.

Using \cite[Theorem 3.2]{HHZ2}, if $\index(I)=\infty$, then $\index(I^k)=\infty$  for any $k>0$. Moreover,  using Theorem \ref{main}, if $\index(I)=1$, then $\index(I^k)=1$  for any $k>0$.

(b) Let $G$ be a simple graph on the vertex set $\{x_{1}, x_{2}, \ldots, x_{n}\}$. The {\em whisker graph} $W(G)$  of $G$ is a simple graph  whose vertex set is $\{x_{1}, x_{2}, \ldots, x_{n}\}\cup\{y_1, y_2, \ldots, y_n\}$, where $y_1, \ldots, y_n$ are new vertices. The edge set of $W(G)$ is $E(G)\cup \{\{x_{i}, y_{i}\}: 1\leq i\leq n\}$.
Furthermore,  by  $ L(G)$ we denote a graph obtained  from $G$ by adding a loop to each of  its vertices and call it the {\em loop graph} of $G$.

Again as a consequence of Theorem~\ref{main}, it follows together with  \cite[Theorem 3.2]{HHZ2} that $I(L(G))^k$ has a linear resolution for all $k$ if and only if $G$ is  complete, and $\index (I(L(G))^k)=1$ for all $k$ if and only if $G$ is not complete. A similar statement holds for $I(W(G))$, because $I(W(G))$ is obtained from $I(L(G))$ by polarization.
\end{Examples}

 \section{Edge ideals of maximal finite index}
In this section we classify   those graphs whose edge ideal has maximal finite index.  In particular our aim is to prove the following result.

\begin{Theorem}
\label{main2}
Let $n\geq 4$, and let  $G$ be a simple graph on the vertex set $[n]$ with no isolated vertices, and let $I$ be  its edge ideal. The following conditions are equivalent:
\begin{itemize}
\item[(a)] The complement $\bar{G}$ of $G$ is a cycle of length $n$;
\item[(b)] $\index(I)=\projdim(I)$.
\end{itemize}
If the equivalent conditions hold, then $\projdim(I)=n-3$.
\end{Theorem}

To prove this theorem we need some intermediate steps. We first observe the following fact which will be used several times in the sequel:

\medskip
Let $G$ be a graph on the vertex set $[n]$ and $\Delta(G)$ be its clique complex. We have
\begin{eqnarray*}
\label{Delta1}
\Delta(G_W)=\Delta(G)_W.
\end{eqnarray*}
In other words,  $\Delta(G)_W$ is the clique complex of induced subgraph of $G$ on the vertex set $W$. Moreover,  $G_W$ is connected if and only if $\Delta(G)_W$ is connected. Here $G_W$ (resp.\ $\Delta(G)_W$) denotes the induced subgraph  of $G$ (resp.\ the induced subcomplex of $\Delta(G)$) whose vertex set is $W$.

\begin{Lemma}
\label{index}
Let $n\geq 4$, and let $G$ be a simple graph on the vertex set $[n]$ and $I$ its edge ideal. Suppose that  $\index(I)=\projdim(I)=t$. Then $\beta_{t,t+2}(I)=0$.
\end{Lemma}

\begin{proof}
Let $\Delta=\Delta(\bar{G})$. By using Hochster's formula \cite[Theorem 8.1.1]{HHBook}, it is enough to show that $\widetilde{H}_{0}(\Delta_{W}; K)= 0$ for any $W\subset [n]$ with $|W|=t+2$. To show this, it is sufficient to prove that $\Delta_{W}$ (equivalently $(\bar{G})_W$) is connected for any $W\subset [n]$ with $|W|=t+2$.

\medskip
Since $\index(I)=t$,  \cite[Theorem 2.1]{EGHP} implies  that there exists a minimal cycle $C$ of length $t+3$ in $\bar{G}$. Without loss of generality we may suppose that $V(C)=\{1, 2, \ldots, t+3\}$ and $E(C)=\{\{i, i+1\} \:\ 1\leq i\leq t+2\}\cup \{\{1, t+3\}\}$.

\medskip
Let $W$ be a subset of $[n]$ with $|W|=t+2$. We consider different cases for $W$ and prove in each case that $(\bar{G})_W$ is connected.

First assume that $W\subset V(C)$.  Since  $V(C)$ has just one vertex more than $W$ we see that $(\bar{G})_W$ is a path
and thus  it is connected.

Now assume that $W\setminus V(C)\neq \emptyset$.
We first claim that for all $j\in W\setminus V(C)$ and all $i\in V(C)$ we have  $\{j, i\}\in E(\bar{G})$.
Indeed, suppose that $\{j, i\}\notin E(\bar{G})$ for some $j\in W\setminus V(C)$ and some $i\in V(C)$. Let  $W'=V(C)\cup \{j\}$ and consider $\Delta_{W'}$.  Note that $\Delta_{W'}$, as a topological space, is homotopy equivalent either to $\mathcal{S}_1$ or to $\mathcal{S}_1$ together with an isolated point. The second case happens only if $\{j, i\}\notin E(\bar{G})$ for all $i\in V(C)$. In either case we see that $\widetilde{H}_{1}(\Delta_{W'}; K)\neq 0$. Now Hochster's formula implies  that $\beta_{t+1,t+4}(I)\neq 0$,   and so $\projdim(I)\geq t+1$, a contradiction. Thus the claim follows.
 Our claim implies that $(\bar{G})_W$ is connected, if $W\cap V(C)\neq \emptyset$.

\medskip
Now suppose that  $W\cap V(C)= \emptyset$. Then $|W\setminus V(C)|=|W|=t+2\geq 3$. Suppose that there exist $j, j'\in W$ such that $\{j, j'\}\notin E(\bar{G})$. Let  $W''=V(C)\cup\{ j, j'\}$. Since $C$ is a minimal cycle and since $j, j'$ are neighbors of all vertices of $C$ we have $\mathcal{F}(\Delta_{W''})=\{\{i, i+1, j\}, \{i, i+1, j'\}:\ 1\leq i\leq t+2\}\cup \{\{1, t+3, j\}, \{1, t+3, j'\}\}$. It follows that $\Delta_{W''}$, as a topological space, is homotopy equivalent to $\mathcal{S}_2$.
  Therefore $\widetilde{H}_2(\Delta_{W''}; K)\neq 0$ and so $\beta_{t+1, t+5}(I)\neq 0$, by Hochster's formula. This implies that $\projdim(I)\geq t+1$, a contradiction.
  So in this case $\{j, j'\}\in E(\bar{G})$ for all $j, j'\in W$. It follows that $(\bar{G})_W$ is a complete graph and so it is connected.
  This completes the proof.
\end{proof}

\begin{Proposition}
\label{beta}
Let $n\geq 4$,  and let $G$ be a simple graph on the vertex set $[n]$ with no isolated vertices,  and let $I$ be its edge ideal. Suppose that  $\index(I)=\projdim(I)=t$.   Then
\begin{itemize}
\item[(a)] $n=t+3$,
\item[(b)] $\beta_{t,t+3}(I)=1$.
\end{itemize}
\end{Proposition}
\begin{proof}
(a) Let $\Delta=\Delta(\bar{G})$. By Lemma \ref{index}, $\beta_{t, t+2}(I)=0$, and so  as a consequence of Hochster's formula,
$\Delta_W$ is connected for any $W\subset [n]$ with $|W|=t+2$.
Since $\index(I)=t$, \cite[Theorem 2.1]{EGHP} implies that there exists a minimal cycle of length $t+3$ in $\bar{G}$, say $C$. We may assume that $V(C)=\{1, 2, \ldots, t+3\}$  and  $E(C)=\{\{i, i+1\} \:\ 1\leq i\leq t+2\}\cup \{\{1, t+3\}\}$.

Assume that $n> t+3$.  We will show that under this assumption,  there exists $W\subset [n]$ such that  either $|W|=t+2$ and $(\bar{G})_W$ is disconnected which implies that $\Delta_W$ is disconnected, or $|W|=t+5$ and $\widetilde{H}_2(\Delta_W; K)\neq 0$ which implies that $\beta_{t+1, t+5}(I)\neq 0$, and so in this case $\projdim(I)>t$. Both cases are not possible, and hence it will follow that $n=t+3$.

For the construction of such $W$ we consider two cases. Let $j\in [n]\setminus [t+3]$.

Suppose first  that there exists $1\leq i\leq t+3$ such that $\{j, i\}\notin E(\bar{G})$.  Let $W=\{j\}\cup V(C) \setminus \{r, s\}$, where $r$ and $s$ are neighbors of $i$ in $C$. So $|W|=t+2$ and  $(\bar{G})_W$ is not connected.

Suppose now that  $\{j, i\}\in E(\bar{G})$ for all $1\leq i\leq t+3$. Assume that  either $[n]\setminus V(C)=\{j\}$ or for all   $j' \in [n]\setminus V(C)$ we have $\{j, j'\}\in E(\bar{G})$. Then  $j$ is an isolated vertex of $G$, a contradiction, since by assumption $G$ has no isolated vertices. So there exists $j'\in [n]\setminus V(C)$ such that $\{j, j'\} \notin E(\bar{G})$.

We may assume that $\{j', i\}\in E(\bar{G})$ for all $1\leq i\leq t+3$, because otherwise, as we have seen before for $j$, there exists $W\subset [n]$ with $|W|=t+2$ such that $(\bar{G})_W$ is not connected.   Now let $W=V(C)\cup\{ j, j'\}$. As we mentioned in the proof of Lemma~\ref{index}, $\widetilde{H}_2(\Delta_{W}; K)\neq 0$ and so $\beta_{t+1, t+5}(I)\neq 0$.

 \medskip
(b) Since $\index(I)=t$,  \cite[Theorem 2.1]{EGHP} implies that there exists a minimal cycle of length $t+3$ in $\bar{G}$, say $C$. Let   $\Delta=\Delta(\bar{G})$.  Then $\widetilde{H}_1(\Delta_{V(C)}; K)\neq 0$.  Hochster's formula implies that $\beta_{t,t+3}(I)\geq1$. Since $n=t+3$, the only $W\subseteq [n]$ with $|W|=t+3$ is $V(C)$, and so $\beta_{t,t+3}(I)=1$, again by Hochster's formula.
\end{proof}

Now we are ready to prove the main theorem of this section.

\begin{proof}[Proof of Theorem~\ref{main2}]
The implication (a)\implies (b) and also $\projdim(I)=n-3$ follows from  \cite[Example 2.2]{EGHP}.

(b)\implies (a):  Let $\index (I)=t$. By  \cite[Theorem 2.1]{EGHP}, $\bar{G}$ contains a minimal cycle of length $t+3$. Proposition~\ref{beta} implies that $G$ has $t+3$ vertices and so does $\bar{G}$. Hence in $\bar{G}$ there are no other vertices. Therefore $\bar{G}$ is a minimal cycle of length $t+3$.
Moreover, $\projdim (I)=n-3$.
\end{proof}

The following result  supports  our conjecture that for a monomial ideal  $I$ generated in degree $2$ one has $\index(I^{k+1})>\index(I^k)$ if $\index(I)>1$.

\begin{Corollary}
\label{maxlinear}
Let $I$ be the edge ideal of a simple graph $G$ and suppose that $I$ has maximal finite index $>1$. Then $\index(I^k)=\infty$ for all $k\geq 2$, i.e. $I^k$ has linear resolution for all $k\geq 2$.
\end{Corollary}

\begin{proof}
We may assume that $G$ has no isolated vertices. By Theorem~\ref{main2} we know that $G$ is  the complement of an $n$-cycle with  $n\geq 5$, in particular $G$ is gap free. We claim that $G$ is claw free. Then by a theorem of Banerjee \cite[Theorem 6.17]{B}, the assertion follows. In order to prove the claim, let $\{i,i+1\}$ for $i=1,\ldots,n-1$ and $\{1,n\}$ be the edges of the cycle $\bar{G}$. Suppose $G$ admits a claw. Then by symmetry we may assume that $\{1,i\}, \{1,j\}$ and $\{1,k\}$ with $1<i<j<k$ are the edges of the claw. However, $\{i,k\}\in E(\bar{G})$, a contradiction.
\end{proof}

\section{Squarefree powers}

Let $I\subset S$ be a squarefree  monomial ideal. Then the {\em $k$-th squarefree power} of $I$, denoted by $I^{[k]}$, is the  monomial ideal generated by all squarefree monomials in $G(I^k)$.

Let $J$ be an arbitrary monomial ideal and let $\alpha=(a_1,a_2,\ldots,a_n)$ be an integer vector with $a_i\geq 0$. Then we let $J_{\leq \alpha}$ be the monomial ideal generated by all monomials $x_1^{c_1}\cdots x_n^{c_n}\in G(J)$ with $c_i\leq a_i$ for $i=1, \ldots, n$.

Now let $\alpha=(1,1,\ldots,1)$. Then $(I^k)_{\leq \alpha}=I^{[k]}$. Therefore it follows from   \cite[Lemma 4.4]{HHZ1}  that $\beta_{i,j}(I^{[k]})\leq \beta_{i,j}(I^k)$ for all $k$. This together with Theorem~\ref{main} implies:
\begin{enumerate}
\item[(i)] $\index(I^{[k]})\geq \index(I^k)$ for all $k$;
\item[(ii)] if $G$ is gap free and $I=I(G)$, then $\index(I^{[k]})>1$ for all $k$.
\end{enumerate}

Here we use the convention that the index of the zero ideal is infinity.

\medskip
The inequality (i) need not be strict. Indeed, if  $I$ is the monomial ideal given by Nevo and Peeva in \cite[Counterexample 1.10]{NP}, then it can be seen, using  computer program, that
$\index(I^k)=\index(I^{[k]})$ for $k=1,\ldots,4$.  On the other hand,  if $G$ is a $9$-cycle, then
$\index(I)=1$, $\index(I^{[2]})=1$, $\index(I^{[3]})=2$ and $\index(I^{[k]})=\infty$ for $k> 3$, while by Theorem~\ref{main}, $\index(I^k)=1$ for all $k$.

\medskip
The converse of (ii) is not true, that is, $G$ may not be gap free but $\index(I^{[k]})>1$ for some $k$. Of course, $\index(I^{[k]})>1$ for $k>n/2$, since for such powers $I^{[k]}=0$. But even if $I^{[k]}\neq 0$ and $G$ is not gap free we may have $\index(I^{[k]})>1$. For example, if $G$ is the graph with vertex set $[4]$ and edges  $\{1,2\},\{3,4\}$, then $G$ is not gap free, but $\index(I(G)^{[2]})=\infty$, because in this case $I(G)^{[2]}=(x_1x_2x_3x_4)$. This and many other examples lead us the Conjecture~\ref{squarefree} below.

\medskip
In the following we assume $G$ admits no isolated vertices.  Recall that a set of edges of $G$ without common vertices is called a {\em matching} of $G$. The {\em matching number} of $G$, denoted $\nu(G)$,  is the maximal size of a matching of $G$. Let $I$ be the edge ideal of $G$. Note that the generators of $I^{[k]}$ correspond bijectively to the set of matchings of $G$ of size $k$.

A matching with the property that one edge in this matching forms a gap with any other edge of this matching will be called a {\em restricted matching}. We denote by $\nu_0(G)$ the maximal size of a restricted matching of $G$. If there is no restricted matching  we set $\nu_0(G)=1$.
Obviously we have
\[
\nu_0(G)\leq \nu(G)=\max\{k\: I^{[k]}\neq 0\}.
\]
For example if $G$ is the whisker graph of a $5$-cycle, then $\nu(G)=5$ and $\nu_0(G)=3$.
\begin{center}
\begin{tikzpicture}[line cap=round,line join=round,>=triangle 45,x=0.6cm,y=0.6cm]
\clip(-1.18,-2.44) rectangle (17.32,4.4);
\draw (12.3,-0.38)-- (14.3,-0.38);
\draw (12.3,-0.38)-- (11.3,1.62);
\draw (11.3,1.62)-- (13.32,3.08);
\draw (13.32,3.08)-- (15.3,1.62);
\draw (15.3,1.62)-- (14.3,-0.38);
\draw [line width=1.6pt] (12.3,-0.38)-- (11.3,-1.38);
\draw (11.3,1.62)-- (10.02,1.62);
\draw [line width=1.6pt] (13.32,3.08)-- (13.32,4.3);
\draw (16.6,1.62)-- (15.3,1.62);
\draw [line width=1.6pt] (15.3,-1.38)-- (14.3,-0.38);
\draw (1.88,-0.34)-- (3.88,-0.34);
\draw (1.88,-0.34)-- (0.88,1.66);
\draw (0.88,1.66)-- (2.9,3.12);
\draw (2.9,3.12)-- (4.88,1.66);
\draw (4.88,1.66)-- (3.88,-0.34);
\draw [line width=1.6pt] (1.88,-0.34)-- (0.88,-1.34);
\draw [line width=1.6pt] (0.88,1.66)-- (-0.4,1.66);
\draw [line width=1.6pt] (2.9,3.12)-- (2.9,4.34);
\draw [line width=1.6pt] (6.18,1.66)-- (4.88,1.66);
\draw [line width=1.6pt] (4.88,-1.34)-- (3.88,-0.34);
\draw (-0.5,-1.72) node[anchor=north west] {\tiny A  matching of maximal size};
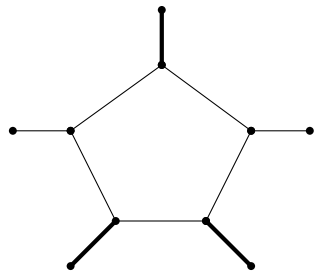
\draw (8.7,-1.8) node[anchor=north west] {\tiny A  matching of maximal size with  gaps};
\begin{scriptsize}
\fill [color=black] (12.3,-0.38) circle (1.5pt);
\fill [color=black] (14.3,-0.38) circle (1.5pt);
\fill [color=black] (11.3,1.62) circle (1.5pt);
\fill [color=black] (13.32,3.08) circle (1.5pt);
\fill [color=black] (15.3,1.62) circle (1.5pt);
\fill [color=black] (11.3,-1.38) circle (1.5pt);
\fill [color=black] (11.3,1.62) circle (1.5pt);
\fill [color=black] (10.02,1.62) circle (1.5pt);
\fill [color=black] (13.32,3.08) circle (1.5pt);
\fill [color=black] (13.32,4.3) circle (1.5pt);
\fill [color=black] (16.6,1.62) circle (1.5pt);
\fill [color=black] (15.3,1.62) circle (1.5pt);
\fill [color=black] (15.3,-1.38) circle (1.5pt);
\fill [color=black] (14.3,-0.38) circle (1.5pt);
\fill [color=black] (1.88,-0.34) circle (1.5pt);
\fill [color=black] (3.88,-0.34) circle (1.5pt);
\fill [color=black] (0.88,1.66) circle (1.5pt);
\fill [color=black] (2.9,3.12) circle (1.5pt);
\fill [color=black] (4.88,1.66) circle (1.5pt);
\fill [color=black] (0.88,-1.34) circle (1.5pt);
\fill [color=black] (0.88,1.66) circle (1.5pt);
\fill [color=black] (-0.4,1.66) circle (1.5pt);
\fill [color=black] (2.9,3.12) circle (1.5pt);
\fill [color=black] (2.9,4.34) circle (1.5pt);
\fill [color=black] (6.18,1.66) circle (1.5pt);
\fill [color=black] (4.88,1.66) circle (1.5pt);
\fill [color=black] (4.88,-1.34) circle (1.5pt);
\fill [color=black] (3.88,-0.34) circle (1.5pt);
\end{scriptsize}
\end{tikzpicture}
\end{center}
\medskip
In general $\nu(G)-\nu_0(G)$ can be arbitrarily large. For example, let $K_n$ be the complete graph on $n$ vertices. Then for its whisker graph $W(K_n)$ we have $\nu(W(K_n))=n$ and $\nu_0(W(K_n))=1$.

On the other hand, let $G$ be an arbitrary tree. We claim that $\nu_0(G)\geq \nu(G)-1$. To see this, let $G$ be an arbitrary graph.  We introduce for each matching  $M$ of $G$ a graph $\Gamma_M(G)$ which we call the {\em matching graph} of $G$. The vertices of $\Gamma_M(G)$ are the elements of $M$. Let $e_1, e_2$ be two elements of $M$ (which are edges of $G$). Then $\{e_1, e_2\}$  is an edge of $\Gamma_M(G)$ if and only if there is another edge $e$ in $G$ such that $e\cap e_1\neq \emptyset$ and $e\cap e_2\neq \emptyset$.

Observe that if $G$ is a tree, then $\Gamma_M(G)$ is a tree. Indeed, suppose that $G$ is a tree and $M$ a matching of $G$. Assume that $\Gamma_M(G)$ contains a  cycle $C$ which we may assume to be minimal. Without loss of generality we may furthermore assume  that $V(C)=\{e_1, e_2, \ldots, e_t\}$ and $E(C)=\{\{e_i, e_{i+1}\}\:\ 1\leq i\leq t-1\}\cup \{\{e_1, e_t\}\}$. Therefore there exist $e'_1, e'_2, \ldots, e'_t \in E(G)$ such that $e'_i\cap e_i\neq \emptyset\neq e'_i\cap e_{i+1}$ for all $1\leq i\leq t-1$ and $e'_t\cap e_t\neq \emptyset\neq e'_t\cap e_{1}$. Assume that $e'_i\cap e_i=\{v_i\}$, $e'_i\cap e_{i+1}=\{w_i\}$ for all $1\leq i\leq t-1$, and $e'_t\cap e_t=\{v_t\}$ and  $e'_t\cap e_{1}=\{w_t\}$. Since $\{e_1, e_2, \ldots, e_t\}$ is a matching, it follows that  for all $i$ and $j$ with $i\neq j$  the edges $e_i$ and $e_{j}$ do not have common vertex.  Thus $\{v_i\}=e'_i\cap e_i\neq e'_j\cap e_{j}=\{v_j\}$ for all $1\leq i\leq t-1$, and $\{v_t\}=e'_t\cap e_t\neq e'_1\cap e_{1}=\{v_1\}$. Similarly $w_i\neq w_j$ for all $i, j$ with $i\neq j$. Suppose that $v_i=w_j$ for some $i, j$. Then $e_i\cap e_{j+1}\neq \emptyset$. This is only possible  if $j=i-1$. Therefore $v_i\neq w_j$ for all $i, j$ with $i-j>1$. Now consider the sequence of vertices $v_1, w_1, v_2, w_2, \ldots, v_t, w_t$ in $V(G)$. Clearly $v_i$ is connected to $w_i$ in $G$ by $e'_i$. Moreover $w_i$ is connected to $v_{i+1}$ in $G$ by $e_{i+1}$, and also $w_t$ is connected to $v_1$ by $e_1$. If $w_i=v_{i+1}$, then $w_i$ is connected to $w_{i+1}$ by $e'_{i+1}$. By removing all $v_{i+1}$ from the above sequence whenever $w_i=v_{i+1}$, we obtain a cycle in $G$, a contradiction.

\medspace
Now suppose that $G$ is a tree and $M$ is a maximal matching of $G$. So $|M|=\nu(G)$.   If $\Gamma_M(G)$ contains an isolated vertex $e$, then $M$ is a restricted matching and hence in this case $\nu_0(G)=\nu(G)$. Suppose that there exists no maximal matching $M$ with the property that $\Gamma_M(G)$ admits an isolated vertex. Since  $\Gamma_M(G)$ is a tree, as we have seen before,   there exists a vertex $e$ in $\Gamma_M(G)$ of degree one. Suppose that $\{e, e'\}\in E(\Gamma_M(G))$. Then $e$ is an isolated vertex in the induced subgraph of $\Gamma_M(G)$ on the vertex set $V(\Gamma_M(G))\setminus \{e'\}$. Hence   $M\setminus \{e'\}$ is a  maximal restricted matching of $G$, and so $\nu_0(G)=\nu(G)-1$.

\medskip
In contrast to the ordinary powers of edge ideals there exists  for any edge ideal $I$ a nonzero squarefree power  of $I$ with linear resolution, as follows from the next result.
\begin{Theorem}
\label{lin.quo}
Let $G$ be a simple graph on the vertex set $[n]$ and $I$ its edge ideal. Then $I^{[\nu(G)]}$ has linear quotients.
\end{Theorem}

\begin{proof}
Let $u_1>u_2>\cdots>u_t$ be the generators of $I^{[\nu(G)]}$ ordered lexicographically induced by $x_1>x_2>\cdots>x_n$ and let $u_j=u^{(j)}_{1}u^{(j)}_{2}\cdots u^{(j)}_{\nu(G)}$ for all $1\leq j\leq t$, where $u^{(j)}_{k}=x_{a}x_{b}$ is a monomial corresponding to an edge $\{a,b\}$ of $G$. Note that each generator $u_j$ corresponds to a maximal matching $m(u_j)$ of $G$ which consists of $\nu(G)$ distinct edges of $G$. Hence, for all $1\leq j\leq t$ and all $1\leq k<k'\leq \nu(G)$, $\gcd(u^{(j)}_{k}, u^{(j)}_{k'})=1$.  
We will show that for all $2\leq i\leq t$, the colon ideal $(u_1, u_2, \ldots, u_{i-1})\: u_i$ is generated by variables. Set $J_i=(u_1, u_2, \ldots, u_{i-1})$. Note that $\{u_k/ \gcd(u_l, u_i) \: 1\leq l\leq i-1\}$ is a set of generators of $J_i\:u_i$, see for example \cite[Propositon 1.2.2]{HHBook}.  Let $l<i$. 
Assume that  $1\leq l\leq i-1$ and $x_{r}x_{s}$ divides $u_l/ \gcd(u_l, u_i)$. If $\{r, s\}\in E(G)$, then $m(u_i)\cup \{\{r, s\}\}$ is a matching of $G$ of size $\nu(G)+1$, a contradiction to the fact that $m(u_i)$ is a maximal matching. Hence  no  pair of variables which divide $u_l/ \gcd(u_l, u_i)$  corresponds to an edge of $G$. 

Suppose $m:=\deg(u_l/ \gcd(u_l, u_i))>1$. We prove that there exists $l'<i$ such that $u_{l'}/\gcd(u_{l'},u_i)$ is of degree one and it divides $u_{l}/\gcd(u_{l},u_i)$. 
 Suppose $u_l/ \gcd(u_l, u_i)=x_{r_1}x_{s_1}x_{s_2}\cdots x_{s_{m-1}}$  and  $x_{r_1}>x_{s_k}$ for all $1\leq k\leq m-1$. Since $\deg u_l=\deg u_i$, it follows that $\deg (u_i/ \gcd(u_i, u_l))=m$. Let  $u_i/ \gcd(u_i, u_l)=x_{a_1}x_{a_2}\cdots x_{a_{m}}$.  Since $u_l>_{lex}u_i$ we have $x_{r_1}>x_{a_k}$ for all $1\leq k\leq m$. As $x_{r_1}$ divides $u_l$, we have  $u^{(l)}_{k_1}=x_{r_1}x_{r_2}$ for some $1\leq r_2\leq n$ and some $1\leq k_1\leq \nu(G)$. Since $\{r_1,r_2\}\in E(G)$ we have $r_2\notin \{s_1,\ldots,s_{m-1}\}$ for the above-mentioned reason. Therefore $x_{r_2}$ divides $u_i$. It follows that there exist ${k_2}$ and $r_3$ with $u^{(i)}_{k_2}=x_{r_2}x_{r_3}$. If $x_{r_1}>x_{r_3}$, then set $u_{l'}:=x_{r_1}u_i/x_{r_3}$. Since $x_{r_1}$ does not divide $u_i$, the monomial $u_{l'}$ corresponds to a matching $m(u_{l'})$ of $G$ with $m(u_{l'})=(m(u_i))\setminus\{\{r_2,r_3\}\})\cup \{\{r_1,r_2\}\}$. Therefore $u_{l'}\in \mathcal{G}(I^{[\nu(G)]})$. Since $x_{r_1}>x_{r_3}$ we have $u_{l'}>_{lex}u_i$ and hence $u_{l'}\in \mathcal{G}(J_i)$. Now $u_{l'}/\gcd(u_{l'},u_i)=x_{r_1}$ and $x_{r_1}|u_l/\gcd(u_l,u_i)$ and hence we are done. 


Now suppose $x_{r_1}<x_{r_3}$. Since $x_{a_k}<x_{r_1}<x_{r_3}$ for all $k$, we conclude that $x_{r_3}|u_l$. Therefore $u^{(l)}_{k_3}=x_{r_3}x_{r_4}$ for some $k_3, r_4$. If $r_4=r_1$, then $\{r_3,r_4\}, \{r_1,r_2\}\in m(u_l)$ implies that $r_3=r_2$ which contradicts  the fact that $\{r_2,r_3\}\in E(G)$. Thus $r_4\neq r_1$ and in particular $k_3\neq k_1$.  
 If $r_4\notin\{s_1,\ldots,s_{m-1}\}$, then $x_{r_4}$ divides $u_i$. In this case $u^{(i)}_{k_4}=x_{r_4}x_{r_5}$ for some $k_4, r_5$. If $k_4=k_2$ then  $\{r_3,r_4\}\in E(G)$ implies that  $r_4=r_2$, and hence $\gcd(u^{(l)}_{k_1}, u^{(l)}_{k_3})\neq 1$, a contradiction.   Thus $k_4\neq k_2$. 
 If $r_5\notin \{a_1,\ldots,a_{m}\}$, then $x_{r_5}|u_l$ and so $u^{(l)}_{k_5}=x_{r_5}x_{r_6}$ for some $k_5, r_6$. If $r_6=r_1$, as above, we conclude that $\gcd(u^{(i)}_{k_2}, u^{(i)}_{k_4})\neq 1$ which is a contradiction. Thus $r_6\neq r_1$ and in particular $k_5\neq k_1$. If $k_5=k_3$, then $\{r_4,r_5\}\in E(G)$ implies that $r_5=r_3$ and hence $\gcd(u^{(i)}_{k_2}, u^{(i)}_{k_4})\neq 1$, a contradiction. Thus $k_5\neq k_1, k_3$.   
 If $r_6\notin\{s_1,\ldots,s_{m-1}\}$ we have $u^{(i)}_{k_6}=x_{r_6}x_{r_7}$ for some $k_6, r_7$. If $k_6=k_2$, then $r_6\in \{r_2,r_3\}$ implies that either $\gcd(u^{(l)}_{k_5}, u^{(l)}_{k_1})\neq 1$ or $\gcd(u^{(l)}_{k_5}, u^{(l)}_{k_3})\neq 1$, a contradiction. Similarly, if $k_6=k_4$, then $r_4=r_6$ which implies $\gcd(u^{(l)}_{k_5}, u^{(l)}_{k_3})\neq 1$ which is again a contradiction. Thus $k_6\neq k_2,k_4$. If $r_7\notin \{a_1,\ldots,a_{m}\}$, we have $u^{(l)}_{k_7}=x_{r_7}x_{r_8}$ for some $k_7, r_8$. This  process is continued if  we have either $r_{2j}\notin\{s_1,\ldots,s_{m-1}\}$ or $r_{2j+1}\notin \{a_1,\ldots,a_{m}\}$. But since $\nu(G)$ is finite, this process must terminate after some finite steps. This means that in some step, say $j\geq 2$, either $r_{2j}\in\{s_1,\ldots,s_{m-1}\}$ or $r_{2j+1}\in \{a_1,\ldots,a_{m}\}$. 

Suppose first that $r_{2j}=s_k$ for some  $1\leq k\leq m-1$.  Now 
$$(m(u_i)\setminus \{\{r_2,r_3\}, \{r_4,r_5\},\ldots,\{r_{2j-2},r_{2j-1}\}\})\cup \{\{r_1,r_2\}, \{r_3,r_4\}, \ldots,\{r_{2j-1},s_k\}\}$$
 is a matching of $G$ of size $\nu(G)+1$. This contradicts the assumption that $\nu(G)$ is the size of a maximal matching in $G$. Therefore $r_{2j}\notin\{s_1,\ldots,s_{m-1}\}$ and hence $r_{2j+1}\in \{a_1,\ldots,a_{m}\}$ for some $j\geq 2$.  
Set $u_{l'}:=x_{r_1}u_i/x_{2j+1}$. Then $u_{l'}$ corresponds to the matching $$(m(u_i)\setminus\{\{r_2,r_3\},\{r_4,r_5\},\ldots,\{r_{2j},r_{2j+1}\}\})\cup \{\{r_1,r_2\},\{r_3,r_4\},\ldots,\{r_{2j-1},r_{2j}\}\}.$$
 Since the size of the above matching is $\nu(G)$ we have $u_{l'}\in \mathcal{G}(I^{[\nu(G)]})$ and since $x_{r_1}>x_{a_k}$ for all $k$, we have $u_{l'}>_{lex}u_i$. Thus  $u_{l'}\in \mathcal{G}(J_i)$ with $u_{l'}/\gcd(u_{l'},u_i)=x_{r_1}$ and $x_{r_1}|u_l/\gcd(u_l,u_i)$. This completes the proof. 
\end{proof}

\medskip
Let $I$ be the edge ideal of a simple graph $G$. Because of Theorem \ref{lin.quo}, $\index(I^{[\nu(G)]})>1$. The question arises which is the smallest integer $k_0$ such that $\index(I^{[k]})>1$ for all $k\geq k_0$. A partial answer to this question is given by the next lemma which implies that $k_0\geq \nu_0(G)$.

\begin{Lemma}
\label{levico}
Let $G$ be a simple graph and $I$ its edge ideal.  Then $\index(I^{[k]})=1$ if  $0<k<\nu_0(G)$.
\end{Lemma}
\begin{proof}
Let $\{e_1,e_2,\ldots, e_{\nu_0(G)}\}$ be a restricted matching of $G$  such that the pairs $e_1,e_i$ form a gap of $G$ for $i=2,\ldots,\nu_0(G)$, and let $u_1,\ldots, u_{\nu_0(G)}\in G(I)$ be the corresponding monomials. Let $0<k<\nu_0(G)$, and $u=u_1u_2\cdots u_{k}$ and $v=u_2u_3\cdots u_{k+1}$. We claim  that $u$ and $v$ are  disconnected in $G^{(u,v)}_{I^{[k]}}$ which then by Corollary~\ref{connected}  yields the desired conclusion.

Let $w\in G^{(u,v)}_{I^{[k]}}$ and suppose that $u_1=x_rx_s$. Since $\lcm(u,v)=u_1u_2\cdots u_{k+1}$, the condition on the edges $e_i$ implies that if $x_r$ or $x_s$ divides $w$, then $u_1$ divides $w$. Thus either $w=v$ or $u_1$ divides $w$. Assume now that $u$ and $v$ are connected in  $G^{(u,v)}_{I^{[k]}}$. Then there exists $w\in G^{(u,v)}_{I^{[k]}}$ with $w \neq v$ and such that $\lcm(w,v)=2k+1$. However, $\lcm(w,v)=2k+2$ since $u_1$ divides $w$, a contradiction.
\end{proof}

We actually expect that $k_0=\nu_0(G)$. Thus we have the following

\begin{Conjecture}
\label{squarefree}
Let $G$ be a simple graph and $I$ its edge ideal. Then  $\index(I^{[k]})>1$ if and only if $k\geq \nu_0(G)$.
\end{Conjecture}

\medskip

In support of our conjecture we prove the following result.

\begin{Theorem}
\label{cycle}
 Let $C_n$ be a cycle of length $n>3$ and $I$ its edge ideal. Then the conjecture holds for $C_n$. More precisely we have
 \begin{itemize}
 \item[(a)] $\nu(C_n)=\lfloor n/2\rfloor$;
 \item[(b)]  $\nu_0(C_n)=\nu(C_n)-1$;
 \item[(c)] If $n$ is even, then the ideal $I^{[\nu_0(C_n)]}$ has linear quotients. If $n$ is odd, then $\index(I^{[\nu_0(C_n)]})=2$.
 \end{itemize}
\end{Theorem}

\medskip
To prove this  theorem  we need  some preliminary steps.

\begin{Lemma}
\label{e-o cycle}
Let $C_n$ be a cycle of  length $n>3$ and $I$ its edge ideal.
\begin{itemize}
\item[(a)] If $n$ is even, then
$$\displaystyle G(I^{[\frac{n}{2}-1]})=\left\{\frac{\prod_{i=1}^nx_i}{x_rx_s}\:\ r<s,\ s-r \text{ odd}\right\}.$$
\item[(b)]
If $n$ is odd, then
$$\displaystyle G(I^{[\frac{n-1}{2}-1]})=\left\{\frac{\prod_{i=1}^nx_i}{x_rx_sx_t}\:\ r<s<t,\  \text{$s-r$ and  $t-s$ odd}\right\}.$$
\end{itemize}
\end{Lemma}

\begin{proof}
(a) Since the generators of $G(I^{[n/2-1]})$ correspond to matchings of $C_n$ of size $n/2-1$ and since any such matching misses exactly two vertices, say $r$ and $s$ with $r<s$, it follows that each component of $C_n\setminus \{r, s\}$ has an even number of vertices. One of the components is $[r+1, s-1]$. Therefore $s-r$ is an odd number.

\medskip
(b) Since the generators of $G(I^{[(n-1)/2-1]})$ correspond to matchings of $C_n$ of size $(n-1)/2-1$ and since any such matching misses exactly three vertices, say $r$, $s$ and $t$ with $r<s<t$, it follows that each component of $C_n\setminus \{r, s, t\}$ has an even number of vertices. Two of the   components are $[r+1, s-1]$ and $[s+1, t-1]$. Therefore $s-r$ and $t-s$ are odd numbers.
\end{proof}

\begin{Lemma}
\label{odd cycle}
Let $C_n$ be a cycle of odd length $n>3$ and $I$ its edge ideal. Then  $$I^{[\frac{n-1}{2}-1]}=I_{\Delta},$$ where  $\Delta$ is the  simplicial complex with  facet set
$$\{ [n]\setminus \{r, s, t\}\: \ \ r<s<t,\ \text{$s-r$ or  $t-s$ even}\}.$$
\end{Lemma}

\begin{proof}
For  $F\subseteq [n]$ we set $\xb_F=\prod_{i\in F}x_i$. Let  $\Delta$ be a simplicial complex with  the set of minimal nonfaces
$$\mathcal{N}(\Delta)=\{F\: \xb_F\in G(I^{[(n-3)/{2}]})\}.$$
Then $I_{\Delta}=I^{[(n-3)/2]}$, and hence $F\subset [n]$ with $|F|=n-3$ belongs to $\Delta$ if and only if $\xb_F\notin G(I^{[(n-3)/{2}]})$. By Lemma~\ref{e-o cycle} this is the case if and only if  $F=[n]\setminus \{r, s, t\}$ for some $r, s, t$ with $r<s<t$ and such  that $s-r$ or $t-s$ is even.

Next we claim that all sets $H\subset [n]$ with $|H|\geq n-2$ are non-faces of $\Delta$. To show this, it suffices to show that each $H\subset [n]$ with $|H|=n-2$ is a non-face of $\Delta$, i.e. \  $\xb_H\in (I^{[(n-3)/{2}]})$.  Let  $H=[n]\setminus \{r, s\}$ with $r<s$. Then $\xb_H\in (I^{[(n-3)/{2}]})$ if and only if there exists a matching of $C_n$ of size $(n-3)/{2}$ whose vertex set does not contain $r, s$.

Removing the vertices $r$ and $s$ from $C_n$ we obtain two paths  $L_1$ and $L_2$ with $|V(L_1)|=k_1$ and $|V(L_2)|=k_2$ and such that $k_1+k_2=n-2$, possibly with one of $k_1, k_2$ equal to zero.  Thus a matching of $C_n$ which avoids the vertices $r$ and $s$ is the same as a matching of $L_1$ and $L_2$. It follows that such a maximal size matching has size $\lfloor k_1/2\rfloor +\lfloor k_2/2\rfloor$. Since $n$ is odd and $k_1+k_2=n-2$, we conclude that one of $k_1, k_2$ is odd and the other one is even. So that in any case $\lfloor k_1/2\rfloor +\lfloor k_2/2\rfloor=(n-3)/2$, as desired.

It remains to be shown that there are no facets $F\in \Delta$ with $|F|\leq n-4$. This fact will follow once we have shown that for any subset $M\subset [n]$ with $|M|=4$ there exists $N=\{r, s, t\}\subset M$ with $r<s<t$ and such that $s-r$ or $t-s$ is even. But this immediately follows from the next lemma.
\end{proof}

\medskip
In order to simplify our discussion we introduce the set
$$\mathcal{S}=\{\{r, s, t\}\:\ \text{$r<s<t$, $s-r$ or $t-s$ even} \}.$$

For this set there are 6 different patterns possible as indicated in the following list:
\begin{center}
(i) $e e e$,\;\; \;\;(ii) $e e o$,\;\; \;\;(iii) $o e e$,\;\;  \;\;(iv) $o o o$,\;\; \;\;(v) $o o e$,\;\; \;\;(vi) $e o o$.
\end{center}
Here $e$ stands for even and $o$ for odd.  For example, (iii) describes  the case, where  $r$ is odd, $s$ is even and $t$ is even.

\medskip
The following observation will be useful in the proof of Proposition~\ref{beta odd}.

\begin{Lemma}
\label{evenodd}
For any  $M=\{t_1,t_2,t_3,t_4\}$ with $1\leq t_1<t_2<t_3<t_4\leq n$. We set $M_i=M\setminus\{t_i\}$.  Let
\[
\mathcal{S}(M)=\{i\:\; M_i\in\mathcal{S}\}.
\]
Then  $\mathcal{S}(M)$ has  $2$ or $4$ elements. More precisely, if $|\mathcal{S}(M)|=2$, then either $\mathcal{S}(M)=\{i, i+1\}$ for some $1\leq i\leq 3$ or $\mathcal{S}(M)=\{1, 4\}$.
\end{Lemma}

\begin{proof}
The set $\mathcal{S}(M)$ consists of $4$ elements, if the even-odd pattern on $M$ is one of the following $e e e e$, $e e e o$, $o e e e$, $e e o o$, $o o e e $, $e o o o$, $o o o e$, $o o o o$.

Otherwise we have
\begin{eqnarray*}
\mathcal{S}(e o e e)=\{1, 2\},\hspace{.3cm}\mathcal{S}(e e o e)=\{3, 4\},
\hspace{.3cm}\mathcal{S}(o e o e)=\{2, 3\},\hspace{.3cm}\mathcal{S}(o e e o)=\{1, 4\},\\
\mathcal{S}(e o e o)=\{2, 3\},\hspace{.3cm}\mathcal{S}(e o o e)=\{1, 4\},
\hspace{.3cm}\mathcal{S}(o e o o)=\{1, 2\},\hspace{.3cm}\mathcal{S}(o o e o)=\{3, 4\}.
\end{eqnarray*}
The assertion of the lemma follows from this list.
\end{proof}

\begin{Proposition}
\label{beta odd}
Let $C_n$ be a cycle of odd length $n>3$ and $I$ its edge ideal. Then  $$\beta_{2,n}(I^{[\frac{n-3}{2}]})\neq 0.$$
\end{Proposition}

\begin{proof}
By Lemma \ref{odd cycle}, $I^{[\frac{n-3}{2}]}=I_{\Delta}$ with
$$\mathcal{F}(\Delta)=\{[n]\setminus \{r, s, t\}\: \ \{r, s, t\}\  \in \mathcal{S} \}.$$
So, by  using Hochster's formula, it is enough to show that $ \widetilde{H}_{n-4}(\Delta; K)\neq 0$.

Let $\partial_j$ be $j$-th chain map in the augmented oriented chain complex $\widetilde{\mathcal{C}}=\widetilde{\mathcal{C}}(\Delta)$ of $\Delta$. The elements $b_F=[i_0, i_1, \ldots, i_{j}]$ with $F=\{i_0, i_1, \ldots, i_{j}\}\in \Delta$ and $i_0<i_1<\cdots<i_{j}$ form a $K$-basis of $\widetilde{\mathcal{C}}_{j}$.  By $({b_F})_{t}$ we denote the basis element $[i_0, i_1, \ldots, i_{t-1}, i_{t+1}, \ldots, i_{j}]$.

We have $\widetilde{H}_{n-4}(\Delta; K)= \Ker \partial_{n-4}/\Im \partial_{n-3}$. Since $\dim \Delta=n-4$, this implies that $\Im \partial_{n-3}=0$.   Set $\sigma(F)=\sum_{t=0}^{j}i_t$. We  let
$$\tau=\sum _{F\in {\mathcal{F}(\Delta)}}(-1)^{\sigma(F)}b_F,$$
and claim that $\tau\in\Ker \partial_{n-4}$. The claim will imply that
$$ \widetilde{H}_{n-4}(\Delta; K)= \Ker \partial_{n-4}\neq 0.$$
 We have
\begin{eqnarray}
\label{partial}
\partial_{n-4}(\tau)&=&\sum_{b_F\in \widetilde{\mathcal{C}}_{n-4}}(-1)^{\sigma(F)}(\sum_{j=0}^{n-4}(-1)^{j}{(b_F)}_j)\\
\nonumber
&=& \sum_{b_G\in \widetilde{\mathcal{C}}_{n-5}}(\sum_{j=0}^{n-4}\sum_{b_F\in \widetilde{\mathcal{C}}_{n-4}\atop (b_F)_j=b_G}(-1)^ {\sigma(F)+j})b_G.
\end{eqnarray}
We will show that for any $b_G\in \widetilde{\mathcal{C}}_{n-5}$, the coefficient $$\alpha_G=\sum_{j=0}^{n-4}\sum_{b_F\in \widetilde{\mathcal{C}}_{n-4}\atop (b_F)_j=b_G}(-1)^ {\sigma(F)+j}$$ of $b_G$ in (\ref{partial}) is zero. This then will imply that $\partial_{n-4}(\tau)=0$, as desired.

Let $G=[n]\setminus M$, where $M=\{t_1, t_2, t_3, t_4\}$ with  $t_1< t_2< t_3< t_4$. We set $G^{(i)} =G\union \{t_i\}$. Let $n_i$ be the position of $t_i$ in $b_{G^{(i)}}$. Thus $({b_{G^{(i)}}})_{n_i}=b_G$ for all $1\leq i\leq 4$.  In order to determine the integers $i$, $1\leq i\leq 4$, with $G^{(i)}\in \Delta$, it is enough to consider $\mathcal{S}(M)$. By Lemma \ref{evenodd}, $\mathcal{S}(M)$  is either $\{1,2,3, 4\}$ or $\{i, i+1\}$ for some $1\leq i\leq 3$ or $\{1, 4\}$.

\medskip
In the following we compute $\alpha_G$  depending on  the set  $\mathcal{S}(M)$.

\medskip
Suppose first that  $\mathcal{S}(M)=\{1,2,3, 4\}$. Then $\alpha_G=\sum_{i=1}^4(-1)^ {\sigma(G^{(i)})+n_i}=0$, because  $(-1)^{\sigma(G^{(i)})+n_i}=-(-1)^{\sigma(G^{(i+1)})+n_{i+1}}$ for any $1\leq i\leq 3$.

Indeed, since all the integers between $t_i$ and $t_{i+1}$ belong to  $G^{(i)}$ as well as to $G^{(i+1)}$, it follows that $n_{i+1}=n_i+r$, where $r=t_{i+1}-t_i-1$. Assume first that  $t_i$  and $t_{i+1}$ both are  even or both are odd. Then  $r$ is odd and
\begin{eqnarray*}
\label{sigma1}
 (-1)^{\sigma(G^{(i)})+n_i}&=&(-1)^{(\sigma(G)+t_i)+n_i}\\
 \nonumber
 &=&(-1)^{\sigma(G)+n_i}(-1)^{t_i}\\
 \nonumber
 &=&(-1)^{\sigma(G)+(n_{i+1}-r)}(-1)^{t_{i+1}}\\
  \nonumber
 &=&(-1)^{(\sigma(G)+t_{i+1})+(n_{i+1}-r)}\\
 \nonumber
 &=&(-1)^{\sigma(G^{(i+1)})+n_{i+1}}(-1)^r=-(-1)^{\sigma(G^{(i+1)})+n_{i+1}},
 \end{eqnarray*}
for all $1\leq i\leq 3$.

 Next assume that one of $t_i, t_{i+1}$ is odd and the other one is even. Then  $r$ is even and
  \begin{eqnarray*}
  \label{sigma2}
 (-1)^{\sigma(G^{(i)})+n_i}&=&(-1)^{(\sigma(G)+t_i)+n_i}\\
 \nonumber
 &=&(-1)^{\sigma(G)+n_i}(-1)^{t_i}\\
 \nonumber
 &=&(-1)^{\sigma(G)+(n_{i+1}-r)}(-1)^{t_{i+1}+1}\\
  \nonumber
 &=&(-1)^{(\sigma(G)+t_{i+1})+(n_{i+1}-r)+1}\\
 \nonumber
 &=&(-1)^{\sigma(G^{(i+1)})+n_{i+1}}(-1)^{r+1}=-(-1)^{\sigma(G^{(i+1)})+n_{i+1}},
 \end{eqnarray*}
 for $1\leq i\leq 3$.

\medskip
Now we assume that $\mathcal{S}(M)=\{i, i+1\}$ for some $1\leq i\leq 3$. Since $$(-1)^{\sigma(G^{(i)})+n_i}=-(-1)^{\sigma(G^{(i+1)})+n_{i+1}}$$ for $1\leq i\leq 3$ as we have seen before, we have $$\alpha_G=(-1)^{\sigma(G^{(i)})+n_i}-(-1)^{\sigma(G^{(i+1)})+n_{i+1}}=0.$$

\medskip
Finally assume that $\mathcal{S}(M)= \{1, 4\}$. Since  $t_2$ and $t_3$ are the only integers between $t_1$ and $t_4$ which do not belong to  $G^{(1)}$ as well as to $G^{(4)}$, we have  $n_4=n_1+r-2$, where $r=t_4-t_1-1$. Moreover, the proof of Lemma~\ref{evenodd} shows that in the case that $\mathcal{S}(M)= \{1, 4\}$, the integers $t_1$ and $t_4$ are both even or both odd. In particular, $r$ is odd.
Consequently
\begin{eqnarray*}
\label{sigma3}
 (-1)^{\sigma(G^{(1)})+n_1}&=&(-1)^{(\sigma(G)+t_1)+n_1}\\
 \nonumber
 &=&(-1)^{\sigma(G)+n_1}(-1)^{t_1}\\
 \nonumber
 &=&(-1)^{\sigma(G)+(n_{4}-r+2)}(-1)^{t_{4}}\\
  \nonumber
 &=&(-1)^{(\sigma(G)+t_{4})+(n_{4}-r)+2}\\
 \nonumber
 &=&(-1)^{\sigma(G^{(4)})+n_{4}}(-1)^{r+2}=-(-1)^{\sigma(G^{(4)})+n_{4}}.
 \end{eqnarray*}
 Therefore
$\alpha_G=(-1)^{\sigma(G^{(1)})+n_1}-(-1)^{\sigma(G^{(4)})+n_{4}}=0$.
Hence $\alpha_G$ is zero in any case and this completes the proof.
\end{proof}

Now we are ready to prove the Theorem \ref{cycle}.

\begin{proof}[Proof of  Theorem~\ref{cycle}]
Let us first discuss the case $n=4, 5$.
 Since there is no restricted matching for cycles of length $4$ and $5$, we have $\nu_0(C_n)=1$. Moreover,      $\nu(C_4)=2$ and $\nu(C_5)=2$. Furthermore,  $I^{[k]}$ has linear quotients for $k\geq \nu_0(C_n)$ for $n= 4$. If $n=5$, then  clearly  $\index(I)=2$.
Therefore in these cases all statements of the theorem hold.

Suppose now that $n>5$.
Without loss of generality we can assume that   $V(C_n)=[n]$ and $E(C_n)=\{\{i, i+1\}\: \ 1\leq i\leq n-1\}\cup \{\{1, n\}\}$.

(a) In the case  $n$ is even the set $T=\{\{1, 2\}, \{3, 4\}, \ldots,\{n-1, n\}\}$ is a matching of maximal size. So  $\nu=|T|=n/2$. In the case that $n$ is odd the set $T'=\{\{1, 2\}, \{3, 4\}, \ldots,\{n-2, n-1\}\}$ is a  matching of maximal size and so $\nu=|T|=(n-1)/2$. Thus in general $\nu=\lfloor n/2\rfloor$.

\medskip
(b) In the case  that $n$ is even the set $T=\{\{1, 2\}, \{4, 5\}, \{6, 7\}, \ldots,\{n-2, n-1\}\}$ is a  matching of maximal size such that $\{1, 2\}$ forms a gap with any other edge in this matching and so $\nu_0(C_n)=|T|=(n-2)/2$. Also in the case that $n$ is odd the set $T'=\{\{1, 2\}, \{4, 5\}, \{6, 7\}, \ldots,\{n-3, n-2\}\}$ is a  matching of maximal size such that $\{1, 2\}$ forms a gap with any other edge in this matching and  so $\nu_0(C_n)=|T'|=(n-3)/2$. Thus   in both cases $\nu_0(C_n)=\nu(C_n)-1$, using part (a).

\medskip
(c) Let $n$ be even. By using Theorem \ref{lin.quo} and part (b), it is enough to show that  $I^{[\nu(C_n)-1]}$ has linear quotients.

Let $u_1>u_2>\cdots>u_r$ be the monomial  generators of $I^{[\nu(C_n)-1]}$ ordered  lexicographically.  We will show that the colon ideal $(u_1, u_2, \ldots, u_{i-1})\:u_i$ is generated by linear forms for any $2\leq i\leq r$. Let $J_i=(u_1, u_2, \ldots, u_{i-1})$. As we mentioned in the proof of Theorem ~\ref{lin.quo}, $\{u_j/ \gcd(u_j, u_i) \: 1\leq j\leq i-1\}$ is a set of generators of $J_i\: u_i$.
 By Lemma~\ref{e-o cycle}, for all $1\leq j\leq r$ we have $u_j=(\prod_{k=1}^nx_k)/(x_{l_j}x_{l'_j})$  for some $l_j< l'_j\leq n$ with $l'_j-l_j$ odd.

 Let $t<i$ and $f_t=u_t/ \gcd(u_t, u_i)$,
and suppose that  two of the integers $l_t, l'_t, l_i, l'_i$   are equal. Then, since $u_t>u_i$,  $f_t=x_{l_i}$ if $l_t\neq l_i$, and $f_t=x_{l'_i}$ if $l_t=l_i$.

Next suppose that no two of the integers $l_t, l'_t, l_i, l'_i$ are equal. Then the integers $l_t, l'_t, l_i, l'_i$ are pairwise different. Thus $f_t=x_{l_i}x_{l'_i}$. If $l'_i\leq n-2$, then let $u_j=(\prod_{j=1}^nx_j)/(x_{l_i}x_{l'_i+2})$. Since $l'_i-l_i$ is odd,  it follows that $u_j\in G(I^{[\nu(C_n)-1]})$. Also $u_j>u_i$ and $f_j=x_{l'_i}\in G(J_i:u_i)$. Therefore $f_j$ divides $f_t$.

Suppose that  $l'_i\geq n-1$. First let  $l'_i=n-1$. Let $u_j=(\prod_{j=1}^nx_j)/(x_{l'_i}x_{n})$. Since $l_i< n-1$, it follows that $u_j>u_i$, and hence $u_j\in G(J_i)$ and $f_j=x_{l_i}\in G(J_i:u_i)$. Thus $f_j$ divides $f_t$.

In the case  that $l'_i=n$, since $u_i$ is not the greatest monomial among  monomial generators of $I^{\nu(C_n)-1}$ we have $l_i\leq n-2$, and since $l'_i-l_i$ is odd, it follows that $l_i\leq n-3$. Let $u_j=(\prod_{j=1}^nx_j)/(x_{l_i+2}x_{l'_i})$. So $u_j>u_i$, $u_j\in G(J_i)$, $f_j=x_{l_i}\in G(J_i:u_i)$ and $f_j$ divides $f_t$.

The above discussion of the various cases shows that $J_i:u_i$ is generated by variables, and so $I^{[\nu(C_n)-1]}$ has linear quotients.

\medskip
Now let $n$ be odd.
We will prove that   $\index(I^{[\nu_0(C_n)]})=2$.  By Proposition~\ref{beta odd},   $\beta_{2, n}(I^{[\nu_0(C_n)]})\neq 0$, and since by part (b) of this theorem,  $I^{[\nu_0(C_n)]}$ is generated in degree $n-3$, it follows that the minimal free resolution of $I^{[\nu_0(C_n)]}$ is not linear at $i=2$ and so $\index(I^{[\nu_0(C_n)]})\leq 2$. Therefore it is enough to show that $\index(I^{[\nu_0(C_n)]})>1$. By using Corollary~\ref{connected}  it is sufficient to prove that for any $u, v \in G(I^{[\nu_0(C_n)]})$  there exists a path in the graph $G^{(u, v)}_{I^{[\nu_0(C_n)]}}$  connecting $u$ and $v$. Clearly,    if  $u, v \in G(I^{[\nu_0(C_n)]})$ with $\deg(\lcm(u, v))\leq (n-3)+1$, then $u$ and $v$ are connected in $G^{(u, v)}_{I^{[\nu_0(C_n)]}}$. Suppose  that $u, v \in G(I^{[\nu_0(C_n)]})$ with $\deg(\lcm(u, v))> (n-3)+1$. By  Lemma~\ref{e-o cycle} we have
$u={(\prod_{i=1}^nx_i)}/{(x_rx_sx_t)}$ and $v={(\prod_{i=1}^nx_i)}/{(x_{r'}x_{s'}x_{t'})}$
where $r<s<t$, $r'<s'<t'$ with $s-r$, $t-s$, $s'-r'$ and $t'-s'$ odd.

\medskip
First suppose that $r=r'$. If $s$ or $t$ belongs to $\{s' , t'\}$,  then   $\deg(\lcm(u, v))=(n-3)+1$, a contradiction. Therefore all the integers $s, t, s', t'$ are pairwise distinct. Without loss of generality we may assume that $s<s'$.  Set $w={(\prod_{i=1}^nx_i)}/{(x_rx_sx_{t'})}$. Then since $s-r$ and $s'-r$ are odd, both $s$ and $s'$ are either even or odd. Since $t'-s'$ is odd, it follows that $t'-s$ is also odd. Thus $w\in G(I^{[\nu_0(C_n)]})$. Moreover $w$ divides $\lcm(u,v)$ and $\deg(\lcm(u,w))=(n-3)+1=\deg(\lcm(v,w))$. Therefore $\{u, w\}, \{w, v\}\in E(G^{(u, v)}_{I^{[\nu_0(C_n)]}})$ and so $u$ and $v$ are connected.

\medskip
For the rest of our discussion we suppose that $r\neq r'$.  We may assume that $r<r'$.

\medskip
First consider the case   $s'=t$. If $t$ is odd (resp. even), then since $t-s$, $s-r$ and $s'-r'$ are odd we conclude that $r$ is odd (resp. even) and $r'$ is even (resp. odd). Let $w={(\prod_{i=1}^nx_i)}/{(x_rx_{r'}x_{t})}$. It is seen that $w\in G(I^{[\nu_0(C_n)]})$, $w$ divides $\lcm(u,v)$ and $\deg(\lcm(u,w))=(n-3)+1=\deg(\lcm(v,w))$. Therefore $\{u, w\}, \{w, v\}\in E(G^{(u, v)}_{I^{[\nu_0(C_n)]}})$. This implies that $u$ and $v$ are connected.

\medskip
 Now consider the case that $s'\neq t$.
  Suppose first that both $r$ and $r'$ are odd (resp. even). Then both $s, s'$ are even (resp. odd),  and  both $t, t'$ are odd (resp. even). If $s'<t$, then let $w={(\prod_{i=1}^nx_i)}/{(x_rx_{s'}x_{t})}$ and $w'={(\prod_{i=1}^nx_i)}/{(x_{r'}x_{s'}x_{t})}$. If $s'>t$, then let $w={(\prod_{i=1}^nx_i)}/{(x_sx_{t}x_{s'})}$ and $w'={(\prod_{i=1}^nx_i)}/{(x_{t}x_{s'}x_{t'})}$. In both cases $w, w'\in G(I^{[\nu_0(C_n)]})$,  they  divide $\lcm(u,v)$, and also $\deg(\lcm(u,w))=\deg(\lcm(w,w'))=\deg(\lcm(w',v))=(n-3)+1$. Therefore $\{u, w\}, \{w, w'\}, \{w', v\}\in E(G^{(u, v)}_{I^{[\nu_0(C_n)]}})$ and so $u$ and $v$ are connected.
Finally suppose that one of the integers $r, r'$ is odd and the other one is even. We may assume that $r$ is odd. Then both $s', t$ are odd, and both $s, t'$ are even. If $s'<t$, then let $w={(\prod_{i=1}^nx_i)}/{(x_rx_{r'}x_{s'})}$ and $w'={(\prod_{i=1}^nx_i)}/{(x_{r}x_{r'}x_{t})}$. If $s'>t$, then let $w={(\prod_{i=1}^nx_i)}/{(x_rx_{s}x_{s'})}$ and $w'={(\prod_{i=1}^nx_i)}/{(x_{r}x_{r'}x_{s'})}$.
Thus in both cases $w, w'\in G(I^{[\nu_0(C_n)]})$,  they  divide $\lcm(u,v)$, and also $\deg(\lcm(u,w))=\deg(\lcm(w,w'))=\deg(\lcm(w',v))=(n-3)+1$. Therefore $\{u, w\}, \{w, w'\}, \{w', v\}\in E(G^{(u, v)}_{I^{[\nu_0(C_n)]}})$. Hence $u$ and $v$ are connected.

\medskip
The above argument shows that in any case $u$ and $v$ are connected in $G^{(u, v)}_{I^{[\nu_0(C_n)]}}$, as desired.
\end{proof}



\begin{thebibliography}{}

\bibitem{B}
A.~Banerjee, \textit{The regularity of powers of edge ideals},  J.~Algebraic Combin., {\bf 41} (2015), 303--321.

\medskip
\bibitem{Ca}
M.~Chardin, \textit{Powers of ideals: Betti numbers, cohomology and regularity},  Commutative algebra,  317--333, Springer, New York, 2013.

\medskip
\bibitem{Ch}
K.~A.~Chandler, \textit{Regularity of the powers of an ideal}, Commun.~Algebra, {\bf 25} (1997), 3773--3776.

\medskip
\bibitem{CHT}
D.~Cutkosky, J.~Herzog, and N.~V.~Trung, \textit{Asymptotic behaviour of the Castelnuovo-Mumford regularity}, Compositio Math., {\bf 118} (1999), 243--261.

 \medskip
\bibitem{Co}
A.~Conca, \textit{Regularity jumps for powers of ideals},  Commutative algebra, Lect.~Notes Pure Appl.
Math. 244 (2006), Chapman $\&$ Hall/CRC, 21--32.


 \medskip
\bibitem{DHS}
H.~Dao, C.~Huneke, and J.~Schweig, \textit{Bounds on the regularity and projective dimension of ideals associated to graphs},  J.~Algebraic  Combin., {\bf 38} (2013), 37--55.


\medskip
\bibitem{EGHP}
D.~Eisenbud, M.~Green, K.~Hulek and S.~ Popescu,   \textit{Restricting linear syzygies: algebra and geometry}, Compos.~Math., {\bf 141} (2005) 1460--1478.

\medskip
\bibitem{BCR}
W.~Bruns, A.~Conca,  and T.~R\"omer,   \textit{Koszul homology and syzygies of Veronese subalgebras}, Math. Ann., {\bf 351} (2011) 761--779.



\medskip
\bibitem{GL1}
M.~L.~Green, and R.~Lazarsfeld, \textit{On the projective normality of complete linear series on an algebraic
curve}, Invent.~Math. {\bf 83}  (1986), 73--90.

\medskip
\bibitem{GL2}
M.~L.~Green, and R.~Lazarsfeld, \textit{Some results on the syzygies of finite sets and algebraic curves}, Compos.~Math. {\bf 67} (1988), 301--314.

\medskip
\bibitem{GPW}
V.~Gasharov, I.~Peeva, and V.~Welker, \textit{The lcm-lattice in monomial resolutions}, Math.~Res.~Lett., {\bf 6} (1999), 521--532.


\medskip
\bibitem{HH}
J.~Herzog, and T.~Hibi, \textit{The depth of powers of an ideal}, Journal of Algebra, {\bf 291} (2005), 534--550.


\medskip
\bibitem{HHBook}
J.~Herzog, and T.~Hibi, \textit{Monomial Ideals}, Graduate Text in Mathematics, Springer, 2011.

\medskip
\bibitem{HHZ1}
J.~Herzog, T.~Hibi, and X.~Zheng, \textit{Dirac's theorem on chordal graphs and Alexander duality}, European J.~Combin., {\bf 25} (2004),  949--960.

\medskip
\bibitem{HHZ2}
J.~Herzog, T.~Hibi, and X.~Zheng, \textit{Monomial ideals whose powers have a linear resolution}, Math.~Scand., {\bf 95} (2004), 23--32.


\medskip
\bibitem{HW}
J.~Herzog, and V.~Welker, \textit{The Betti polynomials of powers of an ideal}, J.~Pure Appl.~Algebra,  {\bf 215}  (2011),   589--596.



 \medskip
\bibitem{N}
E.~Nevo, \textit{Regularity of edge ideals of $C_4$-free graphs via the topology of the lcm-lattice}, J.~Combin.~Theory Ser.~A,  {\bf 118}  (2011),   491--501.


\medskip
\bibitem{NP}
E.~Nevo, and I.~Peeva, \textit{$C_4$-free edge ideals}, J.~Algebraic Combin., {\bf 37} (2013),  243--248.
\end{thebibliography}
\end{document}